\documentclass[12pt,a4paper]{article}
\usepackage[utf8]{inputenc}
\usepackage{amsthm}
\usepackage{amssymb}
\usepackage{amsmath}
\usepackage{enumitem}
\usepackage[T1]{fontenc}
\usepackage{mathtools}
\usepackage{xcolor}
\usepackage{graphicx}
\graphicspath{ {./images/} }
\usepackage{caption}
\usepackage{subcaption}
\usepackage{bm}
\usepackage{authblk}

\newcommand{\floor}[1]{\left \lfloor{ #1 }\right \rfloor}

\title{Planar graphs with the maximum number of induced 4-cycles or 5-cycles}
\date{September 28, 2021}
\author{Michael Savery}

\affil{Mathematical Institute, University of Oxford, Oxford OX2 6GG, UK, and the Heilbronn Institute for Mathematical Research, Bristol, UK\\ \texttt{savery@maths.ox.ac.uk}}

\newtheorem{theorem}{Theorem}
\newtheorem{lemma}[theorem]{Lemma}

\newtheorem{claim}{Claim}
\newtheorem{conjecture}[theorem]{Conjecture}
\newtheorem{question}[theorem]{Question}
\theoremstyle{definition}
\newtheorem{definition}{Definition}

\advance\textwidth2\evensidemargin
\begin{document}

\newcounter{casenum}
  \newenvironment{caseof}{\setcounter{casenum}{1}}{\vskip.5\baselineskip}
  \newcommand{\case}[2]{\vskip.5\baselineskip\par\noindent {\bfseries Case \arabic{casenum}:} #1\\#2\addtocounter{casenum}{1}}

\maketitle

\begin{abstract}
    For large $n$ we determine exactly the maximum numbers of induced $C_4$ and $C_5$ subgraphs that a planar graph on $n$ vertices can contain. We show that $K_{2,n-2}$ uniquely achieves this maximum in the $C_4$ case, and we identify the graphs which achieve the maximum in the $C_5$ case. This extends work in a paper by Hakimi and Schmeichel and a paper by Ghosh, Gy\H{o}ri, Janzer, Paulos, Salia, and Zamora which together determine both maxima asymptotically.
\end{abstract}

\section{Introduction}

An important class of problems in extremal graph theory concerns determining the maximum number of induced copies of a small graph $H$ that can be contained in a graph on $n$ vertices. These questions were first considered by Pippenger and Golumbic in \cite{pippenger} where they showed among other things that for every $k$-vertex graph $H$, the maximum number of induced copies of $H$ in an $n$-vertex graph is asymptotically at least $\frac{n^k}{k^k-k}$. The maximum is now known asymptotically for all graphs $H$ on at most four vertices except the path of length 3, as well as for certain complete partite graphs. See \cite{even-zohal} and its references for a good summary of these results, and see \cite{hatami} and \cite{yuster} for some results for other graphs $H$.

The case where $H$ is a cycle has received particular attention. Pippenger and Golumbic conjectured in their paper that the lower bound stated above is asymptotically tight for cycles of length at least 5. This conjecture was verified for 5-cycles by Balogh, Hu, Lidick\'{y}, and Pfender in \cite{balogh}, where they determined exactly the maximum number of induced 5-cycles in a graph on $n$ vertices when $n$ is large. The best upper bound known for general cycles is due to Kr\'{a}l', Norin, and Volec in \cite{kral}, who showed that for all $n$ and $k\geq 5$, every $n$-vertex graph contains at most $2n^k/k^k$ induced $k$-cycles.

A closely related problem is to determine the maximum number of induced copies of $H$ that can be contained in a planar graph on $n$ vertices. We will write $f_I(n,H)$ for this quantity and $f(n,H)$ for the corresponding quantity when the copies of $H$ do not have to be induced.

We consider in particular the case where $H$ is a small cycle. Hakimi and Schmeichel showed in \cite{hakimi} that $f(n,C_3)=3n-8$ for $n\geq 3$, and $f(n,C_4)=\frac{1}{2}(n^2+3n-22)$ for $n\geq 4$. Since a 3-cycle in a graph is always induced, this gives $f_I(n,C_3)=3n-8$. Also, it is straightforward to see that the complete bipartite graph $K_{2,n-2}$ contains $\frac{1}{2}(n^2-5n+6)$ induced 4-cycles, so $f_I(n,C_4)=\frac{1}{2}n^2+O(n)$, as observed in \cite{ghosh2020maximum}. In \cite{gyori} Gy\H{o}ri, Paulos, Salia, Tompkins, and Zamora determined $f(n,C_5)$ exactly for all $n\geq 5$, and in \cite{ghosh2020maximum} Ghosh, Gy\H{o}ri, Janzer, Paulos, Salia, and Zamora showed that $f_I(n, C_5)=\frac{1}{3}n^2 +O(n)$. 

In \cite{cox2021counting} and \cite{cox2021maximum} Cox and Martin determined $f(n,H)$ asymptotically when $H$ is a small even cycle, showing that $f(n,C_{2k})=\left(\frac{n}{k}\right)^k+o(n^k)$ for $k\in \{3,4,5,6\}$. For $n\equiv 0\pmod{k}$, the graph $G_{n,k}$ defined by replacing every second vertex in a $2k$-cycle with $\frac{n}{k}-1$ copies of that vertex contains $(\frac{n}{k}-1)^k$ induced $2k$-cycles, hence also $f_I(n,C_{2k})= \left(\frac{n}{k}\right)^k+o(n^k)$ for $k\in \{3,4,5,6\}$. Cox and Martin go on to conjecture that $f(n,C_{2k})=\left(\frac{n}{k}\right)^k+o(n^k)$ for all $k\geq 7$, which would similarly determine $f_I(n,C_{2k})$ asymptotically for all $k\geq 7$.

Much less is known for odd cycles of length greater than 5. In the same spirit as above, by evenly blowing up $k$ pairwise non-adjacent vertices in a $(2k+1)$-cycle until the graph has approximately $n$ vertices, one can obtain a lower bound of $\left(\frac{n}{k}\right)^k+o(n^k)$ for $f_I(n,C_{2k+1})$ for all $k\geq 1$. One can obtain a slightly better lower bound for $f(n,C_{2k+1})$ by adding a path through each of the blown up sets of vertices in the graph $G_{n,k}$ defined above to obtain a planar graph containing $2k\left(\frac{n}{k}\right)^k+o(n^k)$ non-induced copies of $C_{2k+1}$. These blow-up constructions were first given in \cite{ghoshp5} and \cite{gyori2020generalized}, and were also considered in \cite{cox2021counting} and \cite{cox2021maximum}. Hakimi and Schmeichel showed in \cite{hakimi} that $f(n,C_{2k+1})=O(n^k)$, and this seems to be the best known upper bound.

For results on $f(n,H)$ for various other graphs $H$ see  \cite{alon}, \cite{cox2021counting}, \cite{eppstein}, \cite{ghoshp5}, \cite{grzesik}, \cite{wood}, and \cite{wormald}. In the general case, Huynh, Joret, and Wood proved a far-reaching result in \cite{huynh2021subgraph} which gives the order of magnitude of $f(n,H)$ for all $H$ in terms of a graph parameter called the `flap-number' of $H$. Very recently, in \cite{liu2021homomorphism}, Liu showed (Corollary 6.1 in that paper) that for all graphs $H$, $f_I(n,H)=\Theta(f(n,H))$, thus determining the order of magnitude of $f_I(n,H)$ for all $H$.

In this paper we will determine $f_I(n,C_4)$ and $f_I(n,C_5)$ exactly for large $n$, and will identify the graphs which exhibit these maxima. Our result for 4-cycles is the following.

\begin{theorem}\label{4-cycles}
For large $n$, $f_I(n,C_4)=\frac{1}{2}(n^2-5n+6)$. Moreover, for large $n$, the only $n$-vertex planar graph which contains $f_I(n,C_4)$ induced 4-cycles is the complete bipartite graph $K_{2,n-2}$.
\end{theorem}

Turning to the 5-cycle case, we first define a family of graphs which we will show to be the only $n$-vertex planar graphs containing $f_I(n,C_5)$ induced 5-cycles for large $n$.

\begin{definition}\label{reqd form def}
A graph $G$ on $n\geq 19$ vertices is \emph{of the required form} if it contains distinct vertices $u_1$, $u_2$, $u_3$, and $w$ such that the remainder of its vertices can be partitioned into sets $A=\{a_1,\dots,a_{|A|}\}$, $B=\{b_1,\dots,b_{|B|}\}$, $C=\{c_1,\dots,c_{|C|}\}$, and $Z=\{z_1,\dots,z_6\}$ such that all of the following hold:
\begin{enumerate}
    \item \begin{enumerate}
        \item if $n\equiv 1 \pmod{3}$, then $|A|=|B|=\frac{n-7}{3}$ and $|C|=
        \frac{n-16}{3}$,
        \item if $n\equiv 2 \pmod{3}$, then $|A|=\frac{n-5}{3}$, $|B|=\frac{n-8}{3}$, and $|C|=\frac{n-17}{3}$, or $|A|=|B|=\frac{n-8}{3}$ and $|C|=\frac{n-14}{3}$,
        \item if $n\equiv 0 \pmod{3}$, then $|A|=\frac{n}{3}-2$, $|B|=\frac{n}{3}-3$, and $|C|=\frac{n}{3}-5$, or $|A|=|B|=\frac{n}{3}-2$ and $|C|=\frac{n}{3}-6$,
    \end{enumerate}
    \item the edge set of $G$ contains the edges $u_1u_2$, $u_1u_3$, $u_2u_3$, $u_1z_1$, $u_2z_1$, $u_1z_2$, $u_2z_4$, $z_1z_2$, $z_1z_3$, $z_1z_4$, $z_2z_3$, $z_3z_4$, $z_2w$, $z_3w$, $z_4w$, $u_2z_5$, $u_3z_5$, $z_5w$, $u_1z_6$, $u_3z_6$, and $z_6w$, and the edges in the three complete bipartite graphs with vertex classes $\{u_1,w\}$ and $A$, $\{u_2,w\}$ and $B$, and $\{u_3,w\}$ and $C$, and
    \item the remaining edges of $G$ are taken from $a_{|A|}z_2$, $z_4b_1$, $b_{|B|}z_5$, $z_5c_1$, $c_{|C|}z_6$, $z_6a_1$, $a_ia_{i+1}$ for $1\leq i\leq |A|-1$, $b_ib_{i+1}$ for $1\leq i\leq |B|-1$, and $c_ic_{i+1}$ for $1\leq i\leq |C|-1$ .
\end{enumerate}
\end{definition}

We refer to the edges in point 3 of Definition \ref{reqd form def} as \emph{optional} edges of a graph of the required form, and if none of these edges are present then we say that $G$ is a \emph{principal graph of the required form}. An illustration of a general graph of the required form is given in Figure \ref{reqd form}. In this illustration red lines represent optional edges. A graph of the required form is clearly planar. It is interesting to note that a graph of the required form in which all optional edges are present is a maximal planar graph.

\begin{figure}[!ht]
    \centering
    \includegraphics[width=0.8\textwidth]{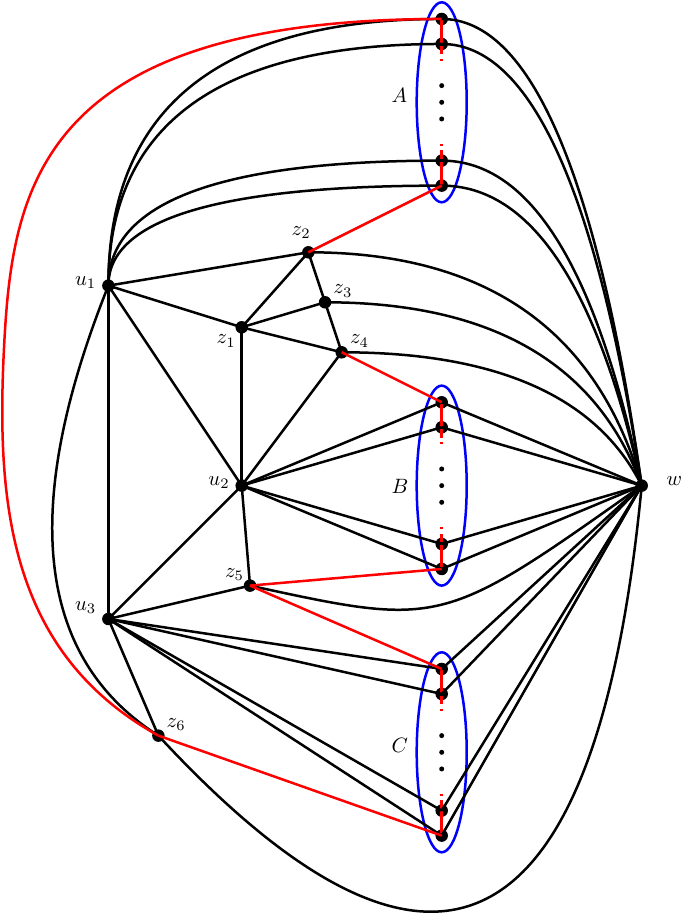}
    \caption{A graph of the required form}
    \label{reqd form}
\end{figure}

Roughly speaking, a large graph of the required form is built from a smaller graph of the required form with $n\equiv 1\pmod{3}$ vertices by ``adding vertices to $A$, $B$, and $C$ as evenly as possible''. If the number of vertices in the smaller graph has some other value modulo 3, then we need to add the first one or two vertices to particular classes before adding the rest as evenly as possible.

In a large graph of the required form almost all of the vertices are in $A$, $B$, or $C$. For each $a\in A$ and $b\in B$ the cycle $u_1awbu_2$ is an induced 5-cycle in the graph. Similarly, there is an induced 5-cycle containing each pair of vertices $b\in B$ and $c\in C$, and each pair $a\in A$ and $c\in C$. Since these 5-cycles are different for different pairs of vertices, and each of $A$, $B$, and $C$ has size $\frac{1}{3}n-O(1)$, this accounts for $3(\frac{1}{3}n-O(1))^2=\frac{1}{3}n^2+O(n)$ induced 5-cycles in $G$.

\begin{theorem}\label{5-cycles}
For large $n$, \[f_I(n,C_5)=\begin{cases}\frac{1}{3}(n^2-8n+22), & \text{if $n\equiv 1\pmod{3}$} \\\frac{1}{3}(n^2-8n+21), & \text{if $n\equiv 0,2 \pmod{3}$}. \end{cases}\] Moreover, for large $n$, if $G$ is a planar graph on $n$ vertices, then $G$ contains $f_I(n,C_5)$ induced 5-cycles if and only if $G$ is of the required form.
\end{theorem}

The conclusion of Theorem \ref{5-cycles} does not hold for all $n$. Indeed, for $n=10$ Theorem \ref{5-cycles} would assert that no planar graph on 10 vertices contains more than 14 induced 5-cycles, but the graph shown in Figure \ref{10 vertex} contains 16. In this figure, the colours are only intended to highlight the structure of the graph.

\begin{figure}[!ht]
    \centering
    \includegraphics[width=0.7\textwidth]{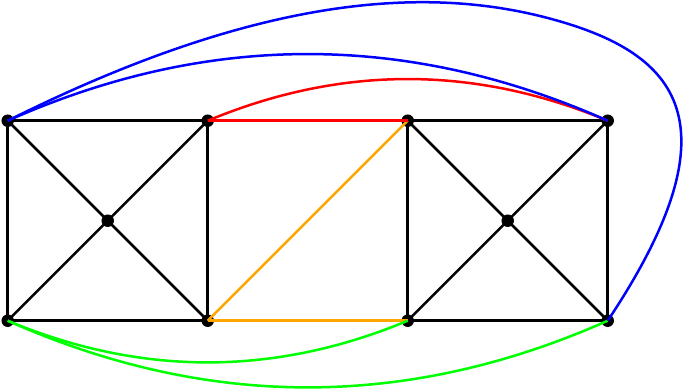}
    \caption{A planar graph with 10 vertices and 16 induced 5-cycles}
    \label{10 vertex}
\end{figure}

For longer cycles we make the following conjecture based on the constructions given above.

\begin{conjecture}
For $k\geq 6$ and $n$ sufficiently large relative to $k$, the $n$-vertex graph obtained by blowing up $\floor{\frac{k}{2}}$ pairwise non-adjacent vertices in a $k$-cycle to sets of as equal size as possible contains the most induced  $k$-cycles of any planar graph.
\end{conjecture}

It would also be interesting to know how many induced $k$-cycles can be contained in an $n$-vertex graph which can be embedded in a fixed surface for surfaces other than the sphere.

\begin{question}\label{other surfaces}
For each $k\geq 3$ and each surface $\Sigma$ other than the sphere, what is the maximum number of induced $k$-cycles which can be contained in an $n$-vertex graph which can be embedded in $\Sigma$?
\end{question}

The results of Huynh, Joret, and Wood \cite{huynh2021subgraph} and Liu \cite{liu2021homomorphism} described above extend to all surfaces $\Sigma$. Combined, these results show that the answer to Question \ref{other surfaces} is $\Theta(n^{\floor{k/2}})$ for all $k\geq 3$ and all surfaces $\Sigma$.

\subsection{Notation and organisation of the paper}
In this paper all graphs are simple, and we use the following standard graph theoretic notation. For a graph $G$ we write $V(G)$ and $E(G)$ for the vertex and edge sets of $G$ respectively. For a vertex $v$ of a graph $G$ we write $N_G(v)$ for the set of neighbours of $v$ in $G$, also called the neighbourhood of $v$ in $G$, and $d_G(v)$ for the degree of $v$ in $G$. In both cases we drop the subscript if the graph in question is clear. For a graph $G$ and a set $S\subseteq V(G)$, we write $G[S]$ for the induced subgraph of $G$ with vertex set $S$, and $G-S$ for the induced subgraph with vertex set $V(G)\setminus S$. We write $K_{a,b}$ for the complete bipartite graph with parts of size $a$ and $b$, and we write $C_k$ for the $k$-cycle graph.

In Section \ref{4-cycle section} we prove Theorem \ref{4-cycles} after stating and proving two preliminary lemmas. Section \ref{5-cycles section} contains the proof of Theorem \ref{5-cycles}, which is broken down into three small preliminary lemmas and two larger lemmas which are proved in Sections \ref{key lemma 1 sect} and \ref{key lemma 2 sect} respectively. The proofs of these two lemmas contain the bulk of the work in proving the theorem. One of these proofs involves some repetitive case checking, which is handled in Appendix \ref{appendix}.

\subsection{Update}
Very recently, the authors of \cite{ghosh2020maximum} updated their paper and independently proved Theorem 2 \cite{ghosh2021maximum}.

\section{Proof of Theorem \ref{4-cycles}}
\label{4-cycle section}
In this section we give two preliminary lemmas and then prove Theorem \ref{4-cycles}.

\begin{lemma}\label{degree 2 then done}
Let $n$ be large, and suppose that $G$ is an $n$-vertex planar graph in which every vertex of $G$ is in at least $n-3$ induced 4-cycles. If $G$ contains a vertex of degree 2, then $G$ is isomorphic to $K_{2,n-2}$.
\end{lemma}
\begin{proof}
Suppose that $v\in V(G)$ has degree 2. Let the neighbours of $v$ be $u$ and $w$. Since $v$ is in at least $n-3$ induced 4-cycles, the remaining $n-3$ vertices of $G$ are all adjacent to both $u$ and $w$, so $u$ and $w$ are adjacent to all vertices of $G$ except each other.

Suppose that $x_1x_2x_3x_4$ is a 4-cycle in $G$ which doesn't contain $u$ or $w$. Then $\{u,x_1,x_3\}$ and $\{w,x_2,x_4\}$ form the partite sets of a subdivision of $K_{3,3}$ in $G$, which contradicts the planarity of $G$ by Kuratowski's theorem. Hence every 4-cycle in $G$ contains $u$ or $w$. Any induced 4-cycle in $G$ containing $u$ must also contain $w$, and vice versa, since no vertex in an induced 4-cycle can be a neighbour of all the others in the cycle. Hence every induced 4-cycle in $G$ has the form $uxwy$ for some $x,y\in V(G)\setminus\{u,w\}$.

Let $z\in V(G)\setminus\{u,w\}$. Then every induced 4-cycle in $G$ containing $z$ is of the form $uzwy$ where $y\in V(G)\setminus\{u,w,z\}$. By assumption, there are at least $n-3$ such cycles, so $z$ is not adjacent to any vertex in $G$ other than $u$ and $w$. Therefore $G$ is isomorphic to $K_{2,n-2}$ as required.
\end{proof}

\begin{lemma}\label{no degree 3,4,5}
Let $n$ be large, and suppose that $G$ is an $n$-vertex planar graph in which every vertex of $G$ is in at least $n-3$ induced 4-cycles. Then $G$ contains no vertices of degree 3, 4, or 5.
\end{lemma}
\begin{proof}
Suppose for a contradiction that $v\in V(G)$ has degree 3, 4, or 5. Let the neighbours of $v$ in $G$ be $x_1,\dots,x_{d(v)}$. The number of induced 4-cycles in $G$ containing $v$ is at most \begin{equation}\label{4-cycle sum}
\sum_{1\leq i<j\leq d(v)}|(N(x_i)\cap N(x_j))\setminus( \{v\}\cup N(v))|.\end{equation}

For each set of three neighbours of $v$, there is at most one common neighbour of all three other than $v$, otherwise $G$ has an obvious $K_{3,3}$ subgraph. So for each distinct pair of terms in sum (\ref{4-cycle sum}), of which there are at most \[\binom{\binom{5}{2}}{2}=45,\] there is at most one vertex contributing to both terms, i.e. there are at most 45 vertices which contribute to more than one term in the sum. Each of these could contribute up to 10 in total, since there are at most 10 terms in the sum, so the total contribution to the sum from these vertices is at most 450.

By assumption $v$ is in at least $n-3$ induced 4-cycles, so there are at least $n-453$ induced 4-cycles containing $v$ in which the vertex opposite $v$ in the cycle is not in any other induced 4-cycle containing $v$. This implies that there are at least $n-453$ vertices in $V(G)\setminus(\{v\}\cup N(v))$ which are adjacent to at least two neighbours of $v$. So there are at most 449 vertices in $V(G)\setminus(\{v\}\cup N(v))$ which are not adjacent to at least two neighbours of $v$.

By the pigeonhole principle, we may assume that $|(N(x_1)\cap N(x_2))\setminus( \{v\}\cup N(v))|\geq (n-3)/10$. Draw $G$ and consider the induced drawing of the complete bipartite graph with parts $\{x_1,x_2\}$ and $(N(x_1)\cap N(x_2))\setminus( \{v\}\cup N(v))$. Label the vertices in the latter set as $y_1,\dots,y_k$ in natural order, where $k\geq (n-3)/10$. The drawing of the complete bipartite graph splits the plane into $k$ regions $R_1,\dots,R_k$, where for each $i$, $R_i$ is bounded by the cycle $x_1y_ix_2y_{i+1}$, where here and henceforth we take subscript addition to be modulo $k$.

Let $R_j$ be the region with $v$ in its interior. Then every vertex in $N(v)$ is in $R_j$ (including its boundary). Hence every vertex with a neighbour in $N(v)\setminus\{x_1,x_2\}$ is in $R_{j-2}\cup R_{j-1}\dots\cup R_{j+2}$ (including its boundary). There are at most 449 vertices in $V(G)\setminus(\{v\}\cup N(v))$ which are not adjacent to at least two neighbours of $v$, so there are at most 449 values of $i$ other than $j-2,\dots,j+2$ for which the interior of $R_i$ contains a vertex. Thus since $n$ is large, there exists $i$ such that $R_{i-1}$ and $R_i$ have no vertices in their interiors.

By assumption $y_i$ is in at least $n-3$ induced 4-cycles. Since $N(y_i)\subseteq\{x_1,x_2,y_{i-1},y_{i+1}\}$, every such 4-cycle either contains the path $y_{i-1}y_iy_{i+1}$ or the path $x_1y_ix_2$. However since $n$ is large (so $k\geq 5$), the existence of $y_{i-2}$ and $y_{i+2}$ implies there is no common neighbour of $y_{i-1}$ and $y_{i+1}$ in $G$ other than $x_1$ and $x_2$, which are both also neighbours of $y_i$. Hence there are no induced 4-cycles containing the path $y_{i-1}y_iy_{i+1}$. Therefore there are $n-3$ common neighbours of $x_1$ and $x_2$ besides $y_i$, and $y_i$ is not adjacent to any of them. Hence $y_i$ has degree 2 in $G$, so by Lemma \ref{degree 2 then done} $G$ is isomorphic to $K_{2,n-2}$. But for large $n$ this has no vertex of degree 3, 4, or 5, which gives the required contradiction and completes the proof of the lemma.
\end{proof}

\begin{proof}[Proof of Theorem \ref{4-cycles}]
First, it is straightforward to see that every induced 4-cycle in $K_{2,n-2}$ contains exactly two vertices from the part of size $n-2$, and that there is a unique, distinct induced 4-cycle containing each such pair, so the number of induced 4-cycles in $K_{2,n-2}$ is $\binom{n-2}{2}=\frac{1}{2}(n^2-5n+6)$.

By Lemmas \ref{degree 2 then done} and \ref{no degree 3,4,5}, to prove the theorem it is sufficient to show that for large $n$ any $n$-vertex planar graph containing $f_I(n,C_4)$ induced 4-cycles has no vertex in fewer than $n-3$ induced 4-cycles. Indeed, if this is the case, then if $n$ is large and $G$ is an $n$-vertex planar graph containing $f_I(n,C_4)$ induced 4-cycles, then the minimum degree of $G$ is at most 5 by the planarity of $G$ and at least 2 since every vertex is in at least one induced 4-cycle. So by the lemmas $G$ is isomorphic $K_{2,n-2}$. 

Let $n$ be large and suppose that $G$ is an $n$-vertex planar graph in which every vertex of $G$ is in more than $n-3$ induced 4-cycles. Then as above the minimum degree of $G$ is at least 2 but at most 5, so by Lemmas \ref{degree 2 then done} and \ref{no degree 3,4,5} $G$ is isomorphic to $K_{2,n-2}$. But this contains vertices in at most $n-3$ induced 4-cycles, which is a contradiction. Thus for large $n$, every $n$-vertex planar graph has a vertex in at most $n-3$ induced 4-cycles.

For large $n$, deleting a vertex in at most $n-3$ induced 4-cycles from an $n$-vertex planar graph containing $f_I(n,C_4)$ induced 4-cycles yields an $(n-1)$-vertex planar graph containing at least $f_I(n,C_4)-n+3$ induced 4-cycles. This shows that $f_I(n,C_4)\leq f_I(n-1,C_4)+n-3$ for large $n$.

We have seen that there are exactly $\frac{1}{2}(n^2-5n+6)$ induced 4-cycles in $K_{2,n-2}$ for each $n$, so there are $n-3$ more induced 4-cycles in $K_{2,n-2}$ than there are in $K_{2,n-3}$. If $f_I(n,C_4)< f_I(n-1,C_4)+n-3$ for infinitely many values of $n$, then for large enough $n$ we have $\frac{1}{2}(n^2-5n+6) > f_I(n,C_4)$, which is a contradiction. Hence for large $n$, $f_I(n,C_4)=f_I(n-1,C_4)+n-3$. Therefore if $n$ is large enough that this holds, and $G$ is an $n$-vertex planar graph containing $f_I(n,C_4)$ induced 4-cycles but also containing a vertex in fewer than $n-3$ induced 4-cycles, then we can delete this vertex to obtain an $(n-1)$-vertex planar graph containing more than $f_I(n-1,C_4)$ induced 4-cycles, which is a contradiction. This completes the proof of the theorem.
\end{proof}

\section{Proof of Theorem \ref{5-cycles}}\label{5-cycles section}
\subsection{Preliminaries to the proof of Theorem \ref{5-cycles}}\label{Prelim sec}
We start with the following definition.
\begin{definition}
Two vertices in a graph $G$ are \emph{principal neighbours} if they are adjacent and there is an induced 5-cycle in $G$ containing both of them.
\end{definition}

We will use the following notation and result adapted from \cite{ghosh2020maximum}. Let $v$ be a vertex of a planar graph $G$ with distinct neighbours $u$ and $w$. Let $X^0_{uvw}=N(u)\setminus(N(w)\cup\{w\})$ and let $Y^0_{uvw}=N(w)\setminus(N(u)\cup\{u\})$. Then let $X_{uvw}$ be the set of vertices in $X^0_{uvw}$ which have a neighbour in $Y^0_{uvw}$, and similarly let $Y_{uvw}$ be the set of vertices in $Y^0_{uvw}$ which have a neighbour in $X^0_{uvw}$. The following important lemma is a small adaptation of Lemma 1 in \cite{ghosh2020maximum}. We include the proof here since it is short, provides intuition, and illustrates a method we will use repeatedly.

\begin{lemma}[\cite{ghosh2020maximum}]\label{X and Y}
Let $G$ be a planar graph with $v\in V(G)$ and $u,w\in N(v)$ such that $u\neq w$. Suppose that there is an induced 5-cycle in $G$ containing the path $uvw$. Define the sets $X_{uvw}$ and $Y_{uvw}$ as above. Then $G'$, the bipartite subgraph of $G$ induced by vertex classes $X_{uvw}$ and $Y_{uvw}$, is a non-empty forest. Moreover the number of induced 5-cycles in $G$ containing the path $uvw$ is at most $|E(G')|$, so in particular, there are at most $|X_{uvw}|+|Y_{uvw}|-1$ such cycles.
\end{lemma}
\begin{proof}
Certainly any induced 5-cycle containing the path $uvw$ contains an edge of $G'$, and each of these edges is in at most one such cycle. Since there is an induced 5-cycle in $G$ containing the path $uvw$, the graph $G'$ is non-empty. Hence it's enough to show that $G'$ is acyclic. Indeed, suppose $G'$ contains a cycle $x_1y_1x_2\dots x_ky_k$ for some $k\geq 2$, $x_i\in X_{uvw}$, and $y_i\in Y_{uvw}$. Then $G$ contains a subdivision of $K_{3,3}$ with vertex classes $\{u,y_1,y_2\}$ and $\{w,x_1,x_2\}$ which is impossible since $G$ is planar.
\end{proof}

Our second preliminary lemma says that every drawing of a principal graph of the required form has a particular structure.

\begin{lemma}\label{drawing of required form}
Let $n\geq 19$, and let $G$ be an $n$-vertex principal graph of the required form. Given a drawing of $G$ and a labelling of the vertices of $G$ consistent with Definition \ref{reqd form def}, we may assume that the boundaries of the faces in the drawing consist of a fixed set of cycles, namely $u_1u_2u_3$, $u_1a_iwa_{i+1}$ for $1\leq i\leq |A|-1$, $u_2b_iwb_{i+1}$ for $1\leq i\leq |B|-1$, $u_3c_iwc_{i+1}$ for $1\leq i\leq |C|-1$, $u_1a_{|A|}wz_2$, $u_1z_1z_2$, $z_1z_2z_3$, $z_2z_3w$, $u_1u_2z_1$, $u_2z_1z_4$, $z_1z_3z_4$, $z_3z_4w$, $u_2z_4wb_1$, $u_2b_{|B|}wz_5$, $u_3z_5wc_1$, $u_2u_3z_5$, $u_3c_{|C|}wz_6$, $u_1z_6wa_1$, and $u_1u_3z_6$.
\end{lemma}
\begin{proof}
Consider the drawing of $G - Z$ induced by the drawing of $G$. It is straightforward to see that we can relabel the vertices within each of $A$, $B$, and $C$ such the faces of this drawing are bounded by the cycles $u_1u_2u_3$, $u_1a_iwa_{i+1}$ for $1\leq i\leq |A|-1$, $u_2b_iwb_{i+1}$ for $1\leq i\leq |B|-1$, $u_3c_iwc_{i+1}$ for $1\leq i\leq |C|-1$, $u_1a_{|A|}wb_1u_2$, $u_2b_{|B|}wc_1u_3$, and $u_3c_{|C|}wa_1u_1$.

It is clear which of these faces contains each of the vertices in $Z$ in the drawing of $G$. One by one, add the vertices in $Z$ (and their edges to the current graph) back to the drawing of $G- Z$ in the order $z_1,\dots,z_6$, keeping track of the faces of the drawing and their boundaries at each step. At each step there is a unique face to which the next vertex can be added, and the cycles forming the boundaries of the faces in the resulting drawing are fixed. After adding $z_6$, the cycles forming the boundaries of the faces are those in the list in the statement of the lemma. Clearly this new labeling of the vertices is still consistent with Definition \ref{reqd form def}.
\end{proof}

Finally, we count the induced 5-cycles in a principal graph of the required form.
\begin{lemma}\label{5-cycle count}
Let $n\geq 19$. Every principal $n$-vertex graph of the required form contains exactly $\frac{1}{3}(n^2-8n+22)$ induced 5-cycles if $n\equiv 1\pmod{3}$, and exactly $\frac{1}{3}(n^2-8n+21)$ otherwise. Moreover, every $n$-vertex graph of the required form contains at least as many induced 5-cycles as a principal $n$-vertex graph of the required form.
\end{lemma}
\begin{proof}
Let $G$ be a principal $n$-vertex graph of the required form, and label its vertices as in Definition \ref{reqd form def}. Let $a\in A$, then the only neighbours of $a$ are $u_1$ and $w$. Define $X=X_{u_1aw}$ and $Y=Y_{u_1aw}$ as above, then the number of induced 5-cycles in $G$ containing $a$ is equal to the number of edges between these two sets. We see that $X=\{u_2,u_3,z_1\}$ and $Y=B\cup C\cup \{z_3,z_4,z_5\}$, and the number of edges between these sets is $|B|+|C|+5$. None of these 5-cycles contain another vertex in $A$, so there are $|A|(|B|+|C|+5)$ induced 5-cycles containing a vertex in $A$.

Now consider $G-A$. By the same method, we see that each vertex in $B$ is in $|C|+5$ induced 5-cycles in $G-A$, and none of these use another vertex in $B$, so there are $|B|(|C|+5)$ induced 5-cycles in $G-A$ containing a vertex in $B$. Applying the method once more to the graph $G-(A\cup B)$ and a vertex in $C$, we find that there are $2|C|$ induced 5-cycles in that graph containing a vertex in $C$.

It remains to count the number of induced 5-cycles in $G[\{u_1,u_2,u_3,w\}\cup Z]$. Let $\Gamma$ be an induced 5-cycle in this graph and suppose that it does not contain $w$. Then the only available neighbours of $z_5$ and $z_6$ are $u_2$ and $u_3$, and $u_1$ and $u_3$ respectively, but both these pairs are adjacent, so neither $z_5$ nor $z_6$ are in $\Gamma$. Then similarly $u_3$ is not in $\Gamma$. Now $z_1$ is adjacent to all the remaining vertices, so it too is not in $\Gamma$. Thus $\Gamma$ is the induced 5-cycle $u_1z_2z_3z_4u_2$.

Now suppose $\Gamma$ contains $w$ and $z_5$. There are exactly five such cycles (two containing the path $wz_5u_2u_1$, two containing $wz_5u_2z_1$, one containing $wz_5u_3u_1$, and none containing any other path of length 4 extending $wz_5$ beyond $z_5$). By a similar count, there are exactly four induced 5-cycles containing $w$ and $z_6$ but not $z_5$. If $\Gamma$ contains $w$ but not $z_5$ or $z_6$, then it must contain $z_2$ and $z_4$, and we see the only such cycle is $u_1z_2wz_4u_2$.

So in total there are exactly $|A\frac{1}{3}(n^2-8n+21)||B|+|A||C|+|B||C|+5(|A|+|B|)+2|C|+11$ induced 5-cycles in $G$, which is $\frac{1}{3}(n^2-8n+22)$ if $n\equiv 1\pmod{3}$ and $\frac{1}{3}(n^2-8n+21)$ otherwise. Finally, we see from this count that no induced 5-cycle in a principal graph of the required form contains two vertices which are the endpoints of an optional edge. Hence every $n$-vertex graph of the required form contains at least as many induced 5-cycles as the principal $n$-vertex graph of the required form that it contains.
\end{proof}

\subsection{Two key lemmas in the proof of Theorem \ref{5-cycles}}
The following two lemmas are the two key steps in the proof of Theorem \ref{5-cycles}.
\begin{lemma}\label{key lemma 1}
Let $n$ be large, and suppose that $G$ is an $n$-vertex planar graph containing $f_I(n,C_5)$ induced 5-cycles. Suppose also that every vertex of $G$ is in at least $\floor{\frac{2(n-1)}{3}}-2$ induced 5-cycles. Then it contains distinct vertices $u_1$, $u_2$, $u_3$, $a$, $b$, $c$, and $w$, such that
\begin{enumerate}
    \item $u_1$, $u_2$, and $u_3$ are all adjacent to one another but none are adjacent to $w$, and
    \item the principal neighbours of $a$, $b$, and $c$ are exactly $u_1$ and $w$, $u_2$ and $w$, and $u_3$ and $w$ respectively.
\end{enumerate}  
\end{lemma}

\begin{lemma}\label{key lemma 2}
Let $n\geq 19$, and suppose that $G$ is an $n$-vertex planar graph containing $f_I(n,C_5)$ induced 5-cycles. Suppose also that it contains distinct vertices $u_1$, $u_2$, $u_3$, $a$, $b$, $c$, and $w$ satisfying the conditions in Lemma \ref{key lemma 1}. Then there exists a vertex of $G$ in at most $\floor{\frac{2(n-1)}{3}}-2$ induced 5-cycles. If moreover $n\equiv 1\pmod{3}$ and every vertex of $G$ is in at least $\floor{\frac{2(n-1)}{3}}-2=\frac{2n-8}{3}$ induced 5-cycles, then $G$ is of the required form and contains exactly $\frac{1}{3}(n^2-8n+22)$ induced 5-cycles.
\end{lemma}

\subsection{Proof of Theorem \ref{5-cycles} given the key lemmas}
Before proving Lemmas \ref{key lemma 1} and \ref{key lemma 2} we use them to prove Theorem \ref{5-cycles}.

\begin{proof}[Proof of Theorem \ref{5-cycles}]
Let $n$ be large, and let $G$ be an $n$-vertex planar graph containing $f_I(n,C_5)$ induced 5-cycles. By Lemmas \ref{key lemma 1} and \ref{key lemma 2}, some vertex of $G$ is in at most $\floor{\frac{2(n-1)}{3}}-2$ induced 5-cycles. By deleting such a vertex from $G$ we find that $f_I(n,C_5)\leq f_I(n-1,C_5)+\floor{\frac{2(n-1)}{3}}-2$ for large $n$.

Let $c(n)$ denote the number of induced 5-cycles in an $n$-vertex principal graph of the required form. By Lemma \ref{5-cycle count} this is well-defined and satisfies $c(n)=c(n-1)+\floor{\frac{2(n-1)}{3}}-2$ for all $n\geq 20$. Suppose that the inequality at the end of the last paragraph is strict for infinitely many values of $n$. Then for large enough $n$, we have $c(n)> f_I(n,C_5)$ which is a contradiction since graphs of the required form are planar. Hence for large $n$, $f_I(n,C_5) = f_I(n-1,C_5)+\floor{\frac{2(n-1)}{3}}-2$.

Therefore, for large $n$, if $G$ is an $n$-vertex planar graph containing $f_I(n,C_5)$ induced 5-cycles, then every vertex of $G$ is in at least $\floor{\frac{2(n-1)}{3}}-2$ induced 5-cycles, otherwise we could delete the vertex in the fewest induced 5-cycles to obtain an $(n-1)$-vertex planar graph containing more than $f_I(n-1,C_5)$ induced 5-cycles. Hence by Lemmas \ref{key lemma 1} and \ref{key lemma 2}, if $n$ is large with $n\equiv 1\pmod{3}$, and $G$ is an $n$-vertex planar graph containing $f_I(n,C_5)$ induced 5-cycles, then $G$ is of the required form. Moreover, for such $n$ we have $f_I(n,C_5)=\frac{1}{3}(n^2-8n+22)$. Hence for large $n$ with $n\equiv 0,2 \pmod{3}$ we have $f_I(n,C_5)=\frac{1}{3}(n^2-8n+21)$, and so by Lemma \ref{5-cycle count}, for large enough $n$ every $n$-vertex graph of the required form contains $f_I(n,C_5)$ induced 5-cycles. To complete the proof of the theorem it is sufficient to show that if $n$ is large with $n\equiv 0,2 \pmod{3}$, and $G$ is an $n$-vertex planar graph containing $f_I(n,C_5)$ induced 5-cycles, then $G$ is of the required form.

\begin{claim}
Let $n$ be large, with $n\equiv 0\pmod{3}$ or $n\equiv 2\pmod{3}$, and let $G$ be an $n$-vertex planar graph containing $f_I(n,C_5)$ induced 5-cycles. Suppose that $G$ contains an $(n-1)$-vertex graph of the required form, $H$, as an induced subgraph. Then $G$ is of the required form.
\end{claim}
\begin{proof}
Let $v$ be the vertex we delete from $G$ to obtain $H$. Let $H_0$ be a principal graph of the required form on $n-1$ vertices contained in $H$. Fix a drawing of $G$ and consider the induced drawings of $H$ and $H_0$. By Lemma \ref{drawing of required form} we can label the vertices of $H$ according to Definition \ref{reqd form def} such that the boundaries of the faces in the drawing of $H_0$ consist of the cycles listed in that lemma.

The only non-triangular faces in the drawing of $H$ have a boundary consisting of a cycle from the list in Lemma \ref{drawing of required form} of the form $u_iywy'$ for some $i\in \{1,2,3\}$ and $y,y'\in A\cup B\cup C\cup Z$. We know that $v$ is contained in at least $\floor{\frac{2(n-1)}{3}}-2$ induced 5-cycles in $G$ so it cannot be contained in a triangular face of $H$. Let the face of $H$ which contains $v$ be bounded by cycle $u_iywy'$, with $i$, $y$, and $y'$ as above. Both $y$ and $y'$ have degree at most 5 in $H$ so by Lemma \ref{X and Y}, since $n$ is large, these cannot be the only principal neighbours of $v$ in $G$. Hence $v$ is adjacent to $u_i$ and $w$. Its only other possible neighbours are $y$ and $y'$, so by the planarity of $G$ it is now clear that if $n\equiv 2\pmod{3}$, then $G$ is of the required form.

If $n\equiv 0\pmod{3}$, then there is some $i\in \{1,2,3\}$ for which this does not imply that $G$ is of the required form. If $n\equiv 0\pmod{3}$, then by the proof of Lemma \ref{5-cycle count}, one of the vertex classes $A$, $B$, and $C$ in $H$ is such that all the vertices in that class are in exactly $\floor{\frac{2(n-2)}{3}}-2=\frac{2n}{3}-4$ induced 5-cycles in $H$. This is strictly less than $\floor{\frac{2(n-1)}{3}}-2=\frac{2n}{3}-3$ and hence each vertex in that class must be in an induced 5-cycle containing $v$ in $G$. Suppose that $A$ is the class to which this applies (similar arguments hold for $B$ and $C$). Then it is enough to show that $i\neq 1$ in the argument above.

Suppose for a contradiction that $v$ is adjacent to $u_1$ and $w$ in $G$. Vertex $a_3$ is in an induced 5-cycle containing $v$ in $G$, and this 5-cycle cannot contain both $u_1$ and $w$, so there is a path of length at most 3 from $a_3$ to $v$ which avoids $u_1$ and $w$. Using the list of cycles forming the boundaries of the faces in the drawing of $H_0$, we can deduce that the only vertices to which there is a path in $H$ of length at most 2 which avoids $u_1$ and $w$ are $a_1$, $a_2$, $a_3$, $a_4$, and $a_5$. Thus $v$ is adjacent to one of these five vertices in $G$. Similarly, since $a_9$ is an induced 5-cycle containing $v$ in $G$, $v$ is adjacent to one of $a_7$, $a_8$, $a_9$, $a_{10}$, and $a_{11}$. However using the list of the boundaries of the faces in the drawing of $H_0$ again, we see that this is impossible. Hence $i\neq 1$ as required, and if $n\equiv 0\pmod{3}$, then $G$ is of the required form.
\end{proof}

To conclude, let $n$ be large with $n\equiv 2\pmod{3}$, and let $G$ be an $n$-vertex planar graph containing $f_I(n,C_5)$ induced 5-cycles. Then there is a vertex in exactly $\floor{\frac{2(n-1)}{3}} - 2$ induced 5-cycles in $G$, and deleting this vertex gives an $(n-1)$-vertex planar graph containing $f_I(n-1,C_5)$ induced 5-cycles. Since $n-1\equiv 1\pmod{3}$, this graph is of the required form. Hence by the claim, $G$ is of the required form. Repeating this argument, we can extend this to $n\equiv 0 \pmod{3}$. This completes the proof of the theorem.
\end{proof}

\section{Proof of Lemma \ref{key lemma 1}}\label{key lemma 1 sect}
Following the authors of \cite{ghosh2020maximum}, we say that in a drawing of a planar graph $G$, an \emph{empty $K_{2,7}$} is a $K_{2,7}$ subgraph of $G$ with parts $\{a_1,a_2\}$ and $\{b_1,\dots,b_7\}$, with $b_1,\dots,b_7$ labelled in a natural order in the drawing, such that in the drawing of $G$ the bounded region with boundary $a_1b_1a_2b_7$ contains exactly the vertices $b_2,\dots,b_6$. Empty $K_{2,7}$'s will be useful in the proof of Lemma \ref{key lemma 1} since, with notation as above, $a_1$ and $a_2$ are the only principal neighbours of $b_4$ in the graph. Indeed, the only other possible neighbours of $b_4$ are $b_3$ and $b_5$, but there is no path of length 3 from $b_3$ to $b_5$ avoiding $a_1$, $a_2$, and $b_4$, so no induced 5-cycle contains the path $b_3b_4b_5$. Also, no induced 5-cycle contains the path $a_ib_4b_j$ for $i\in\{1,2\}$ and $j\in\{3,5\}$ since $a_i$ and $b_j$ are neighbours. Similarly, $a_1$ and $a_2$ are the only principal neighbours of $b_2$, $b_3$, $b_5$, and $b_6$.

The proof of Lemma \ref{key lemma 1} will use the following three results, proved in \cite{ghosh2020maximum} as Lemma 2, Lemma 4, and Corollary 1 respectively (a small correction to their Corollary 1 was given in an updated version of their paper, see \cite{ghosh2021maximumv3}).

\begin{lemma}[\cite{ghosh2020maximum}]\label{Ghosh lemma 2}
Let $n$ be large and let $G$ be an $n$-vertex plane graph in which every vertex is in more than $\frac{11n}{20}$ induced 5-cycles. Then $G$ contains an empty $K_{2,7}$.
\end{lemma}

\begin{lemma}[\cite{ghosh2020maximum}]\label{Ghosh lemma 4}
Let $n$ be large and let $G$ be an $n$-vertex plane graph in which every vertex is in more than $\frac{11n}{20}$ induced 5-cycles. Let $u$ and $w$ be distinct vertices of $G$, and let $v_1,\dots,v_6$ be some of their common neighbours, labelled in a natural order in the drawing of $G$. Suppose that the interior of the bounded region with boundary formed of the cycle $uv_3wv_4$ contains no common neighbours of $u$ and $w$. Then if this region contains a vertex, then it contains at least $n^{1/5}$ vertices.
\end{lemma}

\begin{lemma}[\cite{ghosh2020maximum}, \cite{ghosh2021maximumv3}]\label{Ghosh Cor 1}
Let $n$ be large and let $G$ be an $n$-vertex plane graph in which every vertex is in more than $\frac{11n}{20}$ induced 5-cycles. If $u$ and $w$ are distinct vertices of $G$ with $|N(u)\cap N(w)|\geq 7n^{4/5}$, then $G$ contains an empty $K_{2,7}$ whose part of size 2 is $\{u,w\}$.
\end{lemma}

We now prove Lemma \ref{key lemma 1}. The proof is a very slight adaptation of that of Lemmas 5 and 6 from \cite{ghosh2020maximum}, but we repeat it here for completeness.

\begin{proof}[Proof of Lemma \ref{key lemma 1}]
Fix a drawing of $G$. By Lemma \ref{Ghosh lemma 2} this drawing contains an empty $K_{2,7}$. Let $u$ and $w$ be the vertices in the part of size 2, and let $v$ be the `central' vertex of the seven in the other part. Then $u$ and $w$ are the only principal neighbours of $v$, so in particular $u$ and $w$ are not adjacent. Let $X=X_{uvw}$ and $Y=Y_{uvw}$ be defined as in Section \ref{Prelim sec}. Then by Lemma \ref{X and Y} we have $|X|+|Y|\geq \floor{\frac{2(n-1)}{3}}-1$. Let $G'$ be the induced bipartite graph between $X$ and $Y$.

\begin{claim}\label{2 nbrs in X}
No vertex in $V(G)\setminus(Y\cup\{u\})$ is adjacent to more than two vertices in $X$, and no vertex in $V(G)\setminus(X\cup\{w\})$ is adjacent to more than two vertices in $Y$.
\end{claim}
\begin{proof}
Suppose for a contradiction that $a\in V(G)\setminus(Y\cup\{u\})$ is adjacent to three vertices in $X$, say $x_1$, $x_2$, and $x_3$. By the definition of $X$, $a\neq w$ and $x_1,x_2,x_3\in N(u)$. Moreover, there is a path of length 2 to each of $x_1$, $x_2$, and $x_3$ from $w$, where the middle vertex in the path is in $Y$. Hence there exists a vertex $b\in Y\cup\{w\}$ such that $G$ contains a subdivision of $K_{3,3}$ with parts $\{x_1,x_2,x_3\}$ and $\{u,a,b\}$, contradicting Kuratowski's theorem. This proves the first statement, and the second is similar.
\end{proof}

\begin{claim}\label{kl2 claim}
The maximum degree of $G'$ is at least $n^{5/6}$.
\end{claim}
\begin{proof}
Suppose otherwise. We will start by showing that $u$ and $w$ have a common neighbourhood of size at least $\frac{1}{3}n+o(n)$. Let $x\in X$. By assumption, $x$ is in at least $\floor{\frac{2(n-1)}{3}}-2$ induced 5-cycles in $G$. If an induced 5-cycle contains $u$, $x$, and $w$, then since neither $u$ nor $x$ is adjacent to $w$, and by the definition of $Y$, the 5-cycle must be of the form $uxywa$ for some $y\in Y$ and $a\in N(u)\cap N(w)$. Hence the number of such 5-cycles is at most $d_{G'}(x)|N(u)\cap N(w)|$.

If an induced 5-cycle contains $x$ but not both $u$ and $w$, then each vertex in the cycle which is not adjacent to $x$ is either $w$ or has a path of length at most 3 to $x$ which avoids both $u$ and $w$. For a vertex $z\in X\cup Y$, let $T(z)$ be the number of vertices of $G$ from which there is a path to $z$ of length at most 3 which avoids both $u$ and $w$. Then by Lemma \ref{X and Y} there are at most $\binom{d_G(x)}{2}T(x)$ induced 5-cycles in $G$ containing $x$ but not both $u$ and $w$.

Similar arguments hold if we replace $x$ with a vertex in $Y$, so for all $z\in X\cup Y$ we have \begin{equation}\label{big common nbhd eqn}
\floor{\frac{2(n-1)}{3}}-2\leq d_{G'}(z)|N(u)\cap N(w)| + \binom{d_G(z)}{2}T(z).
\end{equation}To show that $|N(u)\cap N(w)|$ is `large' we will therefore aim to show that there exists $z\in X\cup Y$ such that $d_{G'}(z)$, $d_G(z)$, and $T(z)$ are `small'.

First, recall from the proof of Lemma \ref{X and Y} that $G'$ is planar and hence has at most $|X|+|Y|-1$ edges. So if $l_1$ is the number of vertices of degree at least 3 in $G'$, then $3l_1\leq 2(|X|+|Y|-1)$, which gives $|X|+|Y|-l_1\geq \frac{1}{3}\left(|X|+|Y|+2\right)\geq \frac{1}{3}\left(\floor{\frac{2(n-1)}{3}}+1\right)\geq \frac{n}{5}$. Hence there are at least $\frac{n}{5}$ vertices in $G'$ with degree at most 2.

Next, since $G$ is planar it has at most $3n$ edges. If $l_2$ is the number of vertices of degree at least 61 in $G$, then $61l_2\leq 6n$, so $l_2\leq \frac{n}{10}$. Thus there are at least $\frac{9n}{10}$ vertices in $G$ with degree at most 60.

Finally we turn to $T(z)$. Take a maximal matching in $G'$ between $X$ and $Y$. Label the vertices in this matching as $x_1,\dots,x_k$ and $y_1,\dots,y_k$ where $x_i\in X$ and $y_i\in Y$ for all $i$, where edge $x_iy_i$ is in the matching for all $i$, and where the vertices $x_1,\dots,x_k$ are arranged in that order around $u$ in the drawing of $G$. This matching, along with the edges $ux_i$ and $wy_i$ for each $i$ split the plane into $k$ regions $R_1,\dots,R_k$, where $R_i$ has boundary $ux_iy_iwy_{i+1}x_{i+1}$ for each $i$, with the addition in the subscripts taken modulo $k$. Since the matching is maximal, every vertex in $X\cup Y$ in the interior of region $R_i$ is adjacent to one of the vertices in $X\cup Y$ on its boundary. We have assumed that the maximum degree of $G'$ is less than $n^{5/6}$, so the total number of vertices in $X\cup Y$ in each region $R_i$ (including the vertices on its boundary) is at most $4n^{5/6}$.

Let $q$ be any vertex of $G$ which is not $u$ or $w$. Then there is some $1\leq i\leq k$ such that $q$ is in $R_i$ (including possibly on its boundary). If $z$ is a vertex in $X\cup Y$ such that there is a path in $G$ of length at most 3 from $q$ to $z$ avoiding $u$ and $w$, then $z$ is in $R_{i-4}\cup\dots\cup R_{i+4}$ (including on its boundary) where again the addition in the subscripts is modulo $k$. Hence there are at most $36n^{5/6}$ vertices $z$ for which this can hold, and at most $36n^{11/6}$ pairs $(q,z)$ for which this holds. Let $l_3$ be the number of vertices in $X\cup Y$ with $T(z)\geq 1000n^{5/6}$. Then there are at least $1000l_3n^{5/6}$ pairs $(q,z)$ as above, so $1000l_3n^{5/6}\leq 36n^{11/6}$. Rearranging, we find $l_3\leq \frac{36n}{1000}\leq \frac{n}{20}$.

Combining these three calculations, there are at least $\frac{n}{5}$ vertices $z\in X\cup Y$ with $d_{G'}(z)\leq 2$, of which at most $\frac{n}{10}$ have degree greater than 60 in $G$ and at most $\frac{n}{20}$ have $T(z)\geq 1000n^{5/6}$. So there exists $z\in X\cup Y$ with $d_{G'}(z)\leq 2$, $d_G(z)\leq 60$, and $T(z)<1000n^{5/6}$. Hence by (\ref{big common nbhd eqn}) we have \[|N(u)\cap N(w)|\geq \frac{1}{2}\left(\floor{\frac{2(n-1)}{3}}-2-\binom{60}{2}\cdot1000n^{5/6}\right)= \frac{1}{3}n+o(n).\]

To complete the proof of the claim we will show that $G$ contains at most $\frac{2}{9}n^2+o(n^2)$ induced 5-cycles. By Lemma \ref{5-cycle count}, an $n$-vertex principal graph of the required form contains $\frac{1}{3}n^2+O(n)$ induced 5-cycles, so since $n$ is large this contradicts $G$ containing $f_I(n,C_5)$ induced 5-cycles.

Label the vertices in $N(u)\cap N(w)$ in the order they're arranged around $u$ as $v_1,\dots,v_t$. The drawing of the bipartite graph with parts $\{u,w\}$ and $\{v_1,\dots,v_t\}$ contained in the drawing of $G$ splits the plane into $t$ regions $R_1,\dots,R_t$, where $R_i$ has boundary $uv_iwv_{i+1}$, with the addition in the subscript taken modulo $t$. By Lemma \ref{Ghosh lemma 4}, if any of the interiors of the regions $R_3,\dots,R_{t-3}$ contain a vertex, then they contain at least $n^{1/5}$ vertices, so at most $n^{4/5}$ of these interiors contain a vertex.

Thus for all but $o(n)$ values of $i$, none of the regions $R_{i-3}, R_{i-2}, \dots, R_{i+2}$ have any vertices in their interior, and hence the only principal neighbours of $v_i$ are $u$ and $w$. This implies that $v_i$ is in at most $|X|+|Y|-1\leq n-|N(u)\cap N(w)|\leq \frac{2}{3}n+o(n)$ induced 5-cycles in $G$. Hence deleting all $\frac{1}{3}n+o(n)$ of these vertices $v_i$ removes at most $\frac{2}{9}n^2+o(n^2)$ induced 5-cycles.

Let $H$ be the graph which remains after these vertices have been deleted. Then $H$ contains $\frac{2}{3}n+o(n)$ vertices, and it's enough to show that it contains $o(n^2)$ induced 5-cycles. First, let $S=V(H)\setminus(X\cup Y\cup \{u,w\})$. Then since $|X|+|Y|\geq \floor{\frac{2(n-1)}{3}}-1$, $S$ has size $o(n)$. By Claim \ref{2 nbrs in X}, no vertex in $S$ is adjacent to more than six vertices in $X\cup Y\cup \{u,w\}$. Since $H[S]$ is planar, there is therefore a vertex in $S$ of degree at most 11 in $H$. By Lemma \ref{X and Y} this vertex is in at most $\binom{11}{2}n$ induced 5-cycles in $H$. Hence removing the vertices in $S$ from $H$ one by one, at each stage removing the vertex with the smallest degree, we find that deleting $S$ from $H$ removes only $o(n^2)$ induced 5-cycles.

It remains to show that $H'=G[X\cup Y\cup \{u,w\}]$ contains only $o(n^2)$ induced 5-cycles. Let $\Gamma$ be an induced 5-cycle in $H'$, and suppose it does not contain $u$ or $w$. If $\Gamma$ contains a path that goes from $X$ to $Y$, then back $X$, then back to $Y$, then in a similar way to in the proof of Lemma \ref{X and Y}, this implies the existence of a subdivision of $K_{3,3}$ in $G$. So there are no induced 5-cycles of this form.

Suppose instead that $\Gamma$ is of the form $x_1x_2x_3y_1y_2$ where $x_1,x_2,x_3\in X$ and $y_1,y_2\in Y$. Suppose we are given $x_2$. Then by Claim \ref{2 nbrs in X}, $x_2$ has at most two neighbours in $X$, so there are two choices for $x_1$, and then $x_3$ is determined. By assumption each vertex in $X$ has at most $n^{5/6}$ neighbours in $Y$, so for each choice of $x_1$ there are at most $n^{5/6}$ choices for $y_2$. Finally, each such $y_2$ has at most two neighbours in $Y$. Hence there are only $o(n^2)$ such cycles $\Gamma$. Similarly, there are $o(n^2)$ induced 5-cycles with two vertices in $X$ and three in $Y$.

The final step in the proof of the claim will be to show that there are only $o(n^2)$ induced 5-cycles in $H'$ containing $u$ or $w$. We do these by deleting the vertices in $X\cup Y$ one by one in such an order that at each stage the vertex we delete is in $o(n)$ induced 5-cycles containing $u$ or $w$. The vertex we delete at each step is a vertex in what remains of $X\cup Y$ which has the lowest degree in what remains of $G'$. Since $G'$ is a forest, this vertex has degree at most 1 in $G'$.

Consider a particular step in the deletion process where without loss of generality we are deleting $x\in X$. If $\Gamma$ is an induced 5-cycle in $H'$ containing $x$ and one of $u$ and $w$, then it must be of the form $uxyy'x'$ or $wyxx'y'$ for $x'\in X$ and $y,y'\in Y$. If $x$ has degree 0 in what remains of $G'$, then clearly there are no such cycles. Otherwise let $y$ be the unique neighbour of $x$ in what remains of $Y$. By Claim \ref{2 nbrs in X}, $y$ has at most two neighbours in $Y$, and by assumption each of these has at most $n^{5/6}$ neighbours in $X$. Similarly, $x$ has at most two neighbours in $X$ and each of these has at most $n^{5/6}$ neighbours in $Y$. Thus $x$ is in $o(n)$ induced 5-cycles containing $u$ or $w$ in the current graph, as required.
\end{proof}

Relabel $v$ and $u$ as $a$ and $u_1$ respectively. By Claim \ref{kl2 claim}, without loss of generality there is a vertex $u_2\in X$ with at least $n^{5/6}$ neighbours in $Y$. So the complete bipartite graph with parts $\{u_2,w\}$ and $N(u_2)\cap N(w)$ contains an empty $K_{2,7}$ by Lemma \ref{Ghosh Cor 1}. Let $b_1$, $b_2$, and $b_3$ be the central three vertices of the part of size 7, then $u_2$ and $w$ are the only principal neighbours of each of these vertices. None of $b_1$, $b_2$, and $b_3$ are equal to $u_1$, since $u_1$ and $w$ are not neighbours, and moreover at most two of these vertices are adjacent to $u_1$, otherwise $G$ has a $K_{3,3}$ subgraph. Let $b$ be one of the three which is not adjacent to $u_1$, then $b\in Y$ by the definition of $Y$.

Let $X'=X_{u_2bw}$ and $Y'=Y_{u_2bw}$. Then by Lemma \ref{X and Y}, $|X'|+|Y'|\geq \floor{\frac{2(n-1)}{3}}-1$, so $|(X\cup Y)\cap (X'\cup Y')|\geq \frac{1}{4}n$. Clearly $X\cap Y'$ and $X'\cap Y$ are empty, so $|X\cap X'|+|Y\cap Y'|\geq \frac{1}{4}n$. If $z\in X\cap X'$, then $z$ is a vertex in $X$ which is adjacent to $u_2\in X$, so $|X\cap X'|\leq 2$ by Claim \ref{2 nbrs in X}. Thus $|Y\cap Y'|\geq \frac{1}{5}n$.

Let $z_1$, $z_2$, and $z_3$ be distinct vertices in $Y\cap Y'$. Let $t_1$, $t_2$, and $t_3$ be (not necessarily distinct) vertices in $X'$ such that $z_i$ and $t_i$ are neighbours for each $i=1,2,3$. Without loss of generality, the cycle $u_2t_2z_2wb$ splits the plane into two regions, one of which contains $z_1$ and the other of which contains $z_3$. Again without loss of generality, $u_1$ is in the same region as $z_1$. Since $z_3\in Y$, there is a path of length 2 from $u_1$ to $z_3$, so $u_1$ and $z_3$ have a common neighbour among the vertices in the cycle $u_2t_2z_2wb$. However, $u_2\not\in N(z_3)$ since $z_3\in Y'$, $z_2\not\in N(u_1)$ since $z_2\in Y$, $w\not\in N(u_1)$ since there is an induced 5-cycle in $G$ containing the path $u_1aw$, and $b\not\in N(u_1)$ since $b\in Y$. Hence $t_2\in N(u_1)$, and consequently $t_2\in X$. 

Therefore all but at most two vertices in $Y\cap Y'$ have a neighbour in $X\cap X'$, but we know there are at most two vertices in $X\cap X'$, so there exists one, say $u_3$, with at least $\frac{1}{12}n$ neighbours in $Y\cap Y'$. Then $u_3$ and $w$ have a common neighbourhood of size at least $\frac{1}{12}n$ in $G$, so by Lemma \ref{Ghosh Cor 1} $G$ contains an empty $K_{2,7}$ whose part of size 2 is $\{u_3,w\}$. Let $c$ be the central vertex in the part of size seven. Since $u_2\in X$ and $u_3\in X\cap X'$, we have that $u_1$, $u_2$, and $u_3$ are all mutually adjacent and none are adjacent to $w$. The principal neighbours of $a$, $b$, and $c$ are exactly $u_1$ and $w$, $u_2$ and $w$, and $u_3$ and $w$ respectively, and thus the lemma is proved.
\end{proof}

\section{Proof of Lemma \ref{key lemma 2}}\label{key lemma 2 sect}
We now prove Lemma \ref{key lemma 2}.
\begin{proof}[Proof of Lemma \ref{key lemma 2}]
As in the statement of the lemma, let $n\geq 19$, and let $G$ be an $n$-vertex planar graph containing $f_I(n,C_5)$ induced 5-cycles. Suppose that $G$ contains distinct vertices $u_1$, $u_2$, $u_3$, $a$, $b$, $c$, and $w$ satisfying the conditions in the statement of Lemma \ref{key lemma 1}. Let $A$ be the set of vertices of $G$ whose principal neighbours are exactly $u_1$ and $w$. Let $B$ and $C$ be the corresponding sets for $u_2$ and $w$, and $u_3$ and $w$ respectively. Note that $a\in A$, $b\in B$, and $c\in C$ but $u_1,u_2,u_3,w\not\in A\cup B\cup C$.

Define $H$ to be the graph obtained by deleting from $G$ any edges with one endpoint in $A$ and the other endpoint not in $\{u_1,w\}$, and also any analogous edges for $B$ and $C$. Then $H$ is a planar $n$-vertex graph which, by the definitions of $A$, $B$, and $C$, has the same set of induced 5-cycles as $G$.

Let $Z=V(H)\setminus(A\cup B\cup C\cup \{u_1,u_2,u_3,w\})$. Then the graph $H-Z$ consists of the triangle $u_1u_2u_3$ and the three complete bipartite graphs with parts $\{u_1,w\}$ and $A$, $\{u_2,w\}$ and $B$, and $\{u_3,w\}$ and $C$ (see Figure \ref{H-Z}). It is straightforward to see that $a$ is in exactly $|B|+|C|$ induced 5-cycles in $H$ which avoid $Z$, and analogously for $b$ and $c$. In fact, we see that there are exactly $|A||B|+|A||C|+|B||C|$ induced 5-cycles in $H$ which avoid $Z$.

Let $k_1$, $k_2$, and $k_3$ be the non-negative constants such that $a$, $b$, and $c$ are in exactly $|B|+|C|+k_1$, $|A|+|C|+k_2$, and $|A|+|B|+k_3$ induced 5-cycles in $H$ respectively. Let $T=2(|A|+|B|+|C|)+k_1+k_2+k_3$ be the sum of these numbers. Since $|A|+|B|+|C|=n-|Z|-4$, we have
\begin{equation}\label{eqn for T}
    T-2n+2|Z|+8=k_1+k_2+k_3.
\end{equation}

Vertices $a$, $b$, and $c$ are respectively in $k_1$, $k_2$, and $k_3$ induced 5-cycles in $H$ containing a vertex of $Z$. If an induced 5-cycle in $H$ contains two vertices from $a$, $b$, and $c$, then it does not contain any vertices in $Z$, so the number of induced 5-cycles in $H$ which contain a vertex in $Z$ and one of $a$, $b$, and $c$ is $k_1+k_2+k_3$.

Using a similar approach to that used in the proof of Lemma \ref{drawing of required form}, it is straightforward to see that for any drawing of $H-Z$ there exists a labeling of the vertices of $A$ as $a_1,\dots,a_{|A|}$ and similarly for $B$ and $C$ such that the faces of the drawing have boundaries given by the cycles $u_1u_2u_3$, $u_1a_{|A|}wb_1u_2$, $u_2b_{|B|}wc_1u_3$, $u_3c_{|C|}wa_1u_1$, $u_1a_iwa_{i+1}$ for $1\leq i \leq |A|-1$, $u_2b_iwb_{i+1}$ for $1\leq i \leq |B|-1$, and $u_3c_iwc_{i+1}$ for $1\leq i \leq |C|-1$. Fix a drawing of $H$, consider the induced drawing of $H-Z$, and label the vertices of $A$, $B$, and $C$ as above. We will refer to the faces whose boundaries are the second, third, and fourth of the cycles listed as $F_1$, $F_2$, and $F_3$ respectively (see Figure \ref{H-Z}).

\begin{figure}[!ht]
    \centering
    \includegraphics[width=0.7\textwidth]{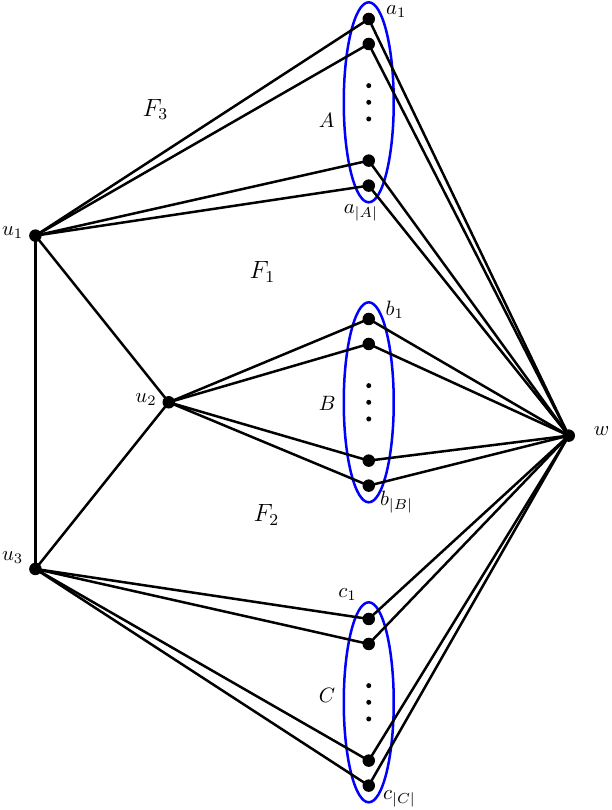}
    \caption{Illustration of the drawing of $H-Z$, with faces $F_1$, $F_2$, and $F_3$ labelled}
    \label{H-Z}
\end{figure}

We now consider which faces of the drawing of $H-Z$ contain vertices of $Z$ in the drawing of $H$. Suppose $|A|\geq 2$, and let $Z_0\subseteq Z$ be the set of vertices in the face with boundary $u_1a_1wa_2$. Then the only vertices in $V(H)\setminus Z_0$ which could be adjacent to a vertex in $Z_0$ are $u_1$ and $w$. The only neighbours of $a_{|A|}$ in $H$ are $u_1$ and $w$, so we can redraw $H$ with no vertices in the face with boundary $u_1a_1wa_2$ by moving the vertices in $Z_0$ into $F_1$. Repeating this process for the other faces of $H-Z$ whose boundaries are 4-cycles, we see that way may assume that all vertices of $Z$ are in $F_1$, $F_2$, $F_3$, or the triangular face.

If $z\in Z$ is in an induced 5-cycle in $H$ containing one of $a$, $b$, and $c$, then $z$ is at distance 1 or 2 from $w$. Therefore $z$ is not in the triangular face of $H-Z$, so it is in one of $F_1$, $F_2$, or $F_3$. Note that no induced 5-cycle in $H$ containing one of $a$, $b$, and $c$ can contain vertices in $Z$ from two of $F_1$, $F_2$, and $F_3$. Hence the number of induced 5-cycles in $H$ containing one of $a$, $b$, and $c$ and a vertex in $Z$ is the sum over $F_1$, $F_2$, and $F_3$ of the number of such cycles containing a vertex of $Z$ in that face.

Let $Z_1\subseteq Z$ be the set of vertices in $F_1$. For each $S\subseteq \{u_1,u_2,w\}$, let $L_S$ be the set of vertices in $Z_1$ whose neighbours among $u_1$, $u_2$, and $w$ are exactly the vertices in $S$, and let $l_S=|L_S|$. The number of induced 5-cycles containing $c$ and a vertex in $Z_1$ is $2l_{\{u_1,u_2,w\}}+l_{\{u_1,w\}} +l_{\{u_2,w\}}$. The number of induced 5-cycles containing $b$ and a vertex in $Z_1$ is $l_{\{u_1,w\}}$ plus the number of paths $u_2z_1z_2w$ where $z_1, z_2\in Z_1$ are such that $z_1\not\in N(w)$ and $z_2\not\in N(u_2)$. The number of such paths is the number of edges between $L_{\{u_2\}}\cup L_{\{u_1,u_2\}}$ and $L_{\{w\}}\cup L_{\{u_1,w\}}$. By a similar argument to that used in the proof of Lemma \ref{X and Y}, the bipartite graph between these two sets is a forest. Moreover, it is straightforward to use the planarity of $H$ to show that there can be at most one vertex in $L_{\{u_1,w\}}$ which has a neighbour in $L_{\{u_2\}}\cup L_{\{u_1,u_2\}}$.

Hence the number of paths $u_2z_1z_2w$ is at most $l_{\{u_2\}}+ l_{\{u_1,u_2\}} + l_{\{w\}}$, with equality only if one of the following holds:
\begin{enumerate}[label=(\roman*)]
    \item there is a vertex in $L_{\{u_1,w\}}$ which has a neighbour in $L_{\{u_2\}}\cup L_{\{u_1,u_2\}}$, every vertex in $L_{\{u_1,u_2\}}$ has a neighbour in $L_{\{w\}}\cup L_{\{u_1,w\}}$, and every vertex in $L_{\{w\}}$ has a neighbour in $L_{\{u_2\}}\cup L_{\{u_1,u_2\}}$, or
    \item $l_{\{u_2\}}=l_{\{u_1,u_2\}}=l_{\{w\}}=0$.
\end{enumerate}

Indeed, if there is no vertex in $L_{\{u_1,w\}}$ with a neighbour in $L_{\{u_2\}}\cup L_{\{u_1,u_2\}}$, then since the bipartite graph between $L_{\{u_2\}}\cup L_{\{u_1,u_2\}}$ and $L_{\{w\}}$ is a forest it contains at most $l_{\{u_2\}}+ l_{\{u_1,u_2\}} + l_{\{w\}}$ edges, with equality if and only if it has no vertices. If instead there is a vertex $p\in L_{\{u_1,w\}}$ which has a neighbour in $L_{\{u_2\}}\cup L_{\{u_1,u_2\}}$, then the number of edges in the bipartite graph between $L_{\{u_2\}}\cup L_{\{u_1,u_2\}}$ and $L_{\{w\}}\cup L_{\{u_1,w\}}$ is the same as the number of edges between $L_{\{u_2\}}\cup L_{\{u_1,u_2\}}$ and $L_{\{w\}}\cup \{p\}$. Again, this is at most $l_{\{u_2\}}+ l_{\{u_1,u_2\}} + l_{\{w\}}$ with equality if and only if the forest is connected, which in particular implies that (i) is satisfied.

So the number of induced 5-cycles containing $b$ and a vertex in $Z_1$ is at most $l_{\{u_1,w\}}+l_{\{u_2\}}+l_{\{u_1,u_2\}} + l_{\{w\}}$ with equality only if one of (i) or (ii) holds. We can similarly bound the number of induced 5-cycles containing $a$ and a vertex in $Z_1$ by swapping the roles of $u_1$ and $u_2$. Let the conditions corresponding to (i) and (ii) in this setting be called (i)' and (ii)'. Summing, we find that the total number of induced 5-cycles that contain one of $a$, $b$, and $c$ and a vertex in $Z_1$ is at most $2l_{\{u_1,u_2,w\}}+2l_{\{u_1,w\}} +2l_{\{u_2,w\}}+ l_{\{u_2\}}+ 2l_{\{u_1,u_2\}} + 2l_{\{w\}} + l_{\{u_1\}}$, with equality only if one of each of (i) and (ii), and (i)' and (ii)' hold.

Repeating for $F_2$ and $F_3$, and recalling that no vertices in the triangular face of $H-Z$ are in an induced 5-cycle in $H$ containing any of $a$, $b$, and $c$, we find that there are at most twice as many induced 5-cycles in $H$ containing one of $a$, $b$, and $c$ and a vertex in $Z$ as there are vertices in $Z$. Recall that the number of such cycles is $k_1+k_2+k_3$, which is equal to $T-2n+2|Z|+8$ by equation (\ref{eqn for T}). Hence \begin{equation}\label{ineq for T}
    T\leq 2n-8.
\end{equation}

Suppose that all the vertices in $G$, so in particular $a$, $b$, and $c$, are in at least $\floor{\frac{2(n-1)}{3}}-1$ induced 5-cycles. Then this is also true in $H$, so \[T\geq3\floor{\frac{2(n-1)}{3}}-3=\begin{cases}2n-6, & \text{if $n\equiv 0\pmod{3}$} \\ 2n-5, & \text{if $n\equiv 1 \pmod{3}$} \\ 2n-7, & \text{if $n\equiv 2 \pmod{3}$}. \end{cases}\] This is a contradiction, so there is a vertex of $G$ in at most $\floor{\frac{2(n-1)}{3}}-2$ induced 5-cycles, which proves the first part of the lemma.

We now prove the second part of the lemma, so assume that $n\equiv 1\pmod{3}$ and every vertex of $G$ is in at least $\floor{\frac{2(n-1)}{3}}-2=\frac{2n-8}{3}$ induced 5-cycles. Then the same is true in $H$, so $T\geq 2n-8$. So in fact we have $T=2n-8$, and each of $a$, $b$, and $c$ are in exactly $\frac{2n-8}{3}$ induced 5-cycles. Moreover, the fact that we have equality in (\ref{ineq for T}) implies that there are exactly twice as many induced 5-cycles in $H$ containing one of $a$, $b$, and $c$ and a vertex in $Z$ as there are vertices in $Z$. Hence there are no vertices of $Z$ in the triangular face of $H-Z$, and in $F_1$ we must have $l_\emptyset=l_{\{u_2\}}=l_{\{u_1\}}=0$ and one of each of conditions (i) and (ii), and (i)' and (ii)' must hold (and analogously in $F_2$ and $F_3$).

We now use these conditions on the vertices in $Z_1$ to show that $H[Z_1\cup\{u_1,u_2,w\}]$ must belong to one of three families of graphs. The same will hold for faces $F_2$ and $F_3$ by symmetry. The first possibility is that $Z_1=\emptyset$, in which case we will say that $F_1$ is \emph{of type 1}. The second possibility is that $Z_1$ contains a single vertex $z$ which is adjacent to all of $u_1$, $u_2$, and $w$, in which case we will say that $F_1$ is \emph{of type 2}.

The final possibility is that $Z_1$ contains $m\geq 3$ vertices $z_1,\dots,z_m$ and $H[Z_1\cup\{u_1,u_2,w\}]$ contains the edges $u_1u_2$, $u_1z_1$, $u_2z_1$, $u_1z_2$, and $u_2z_m$, and the edges $z_1z_i$ and $z_iw$ for all $2\leq i\leq m$. The remaining edges of $H[Z_1\cup\{u_1,u_2,w\}]$ form a subset of $\{z_iz_{i+1}:2\leq i\leq m-1\}$. In this case we say that $F_1$ is \emph{of type 3}. The three possibilities for $F_1$ are illustrated in Figure \ref{Z types}, where red lines indicate edges each of whose presence does not affect whether the face is of that type.

\begin{figure}
    \centering
    \begin{subfigure}[h]{0.3\textwidth}
        \centering
        \includegraphics[width=\textwidth]{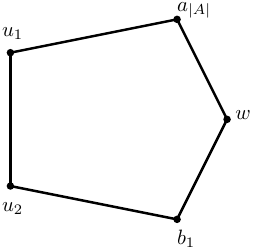}
        \caption{Type 1}
        \label{Z type 1}
     \end{subfigure}
     \hfill
     \begin{subfigure}[h]{0.3\textwidth}
        \centering
        \includegraphics[width=\textwidth]{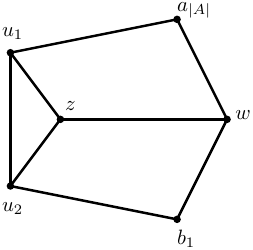}
        \caption{Type 2}
        \label{Z type 2}
     \end{subfigure}
    \hfill
     \begin{subfigure}[h]{0.3\textwidth}
        \centering
        \includegraphics[width=\textwidth]{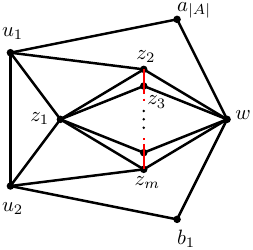}
        \caption{Type 3}
        \label{Z type 3}
     \end{subfigure}
        \caption{The possible configurations of vertices in $F_1$}
        \label{Z types}
\end{figure}

\begin{claim}\label{types claim}
Face $F_1$ is of type 1, 2, or 3. By symmetry, the same is true for $F_2$ and $F_3$.
\end{claim}
\begin{proof}
We know that $l_\emptyset=l_{\{u_2\}}=l_{\{u_1\}}=0$ and one of each of conditions (i) and (ii), and (i)' and (ii)' hold. First, if condition (ii) holds, then $l_{\{u_1,u_2\}}=l_{\{w\}}=0$. So all vertices in $Z_1$ are adjacent to exactly the vertices in one of $\{u_1,w\}$, $\{u_2,w\}$, and $\{u_1,u_2,w\}$ among $u_1$, $u_2$, and $w$. In particular all the vertices are adjacent to $w$. Suppose there exists $z\in L_{\{u_1,w\}}$. Then since $z$ is not in $A$, there is some $z'\in Z_1$ which is a principal neighbour of $z$.

Let $\Gamma$ be an induced 5-cycle in $H$ containing both $z$ and $z'$. Then $\Gamma$ does not contain $w$ since it's a common neighbour of $z$ and $z'$. If every vertex in $\Gamma$ is in $Z_1$, then all its vertices are adjacent to $u_1$ or $u_2$ so without loss of generality, some two non-adjacent vertices in $\Gamma$, say $p_1$ and $p_2$, are both adjacent to $u_1$. Let $q_1$ and $q_2$ be two other distinct non-adjacent vertices of $\Gamma$. Then $\{w,p_1,p_2\}$ and $\{u_1,q_1,q_2\}$ form the partite sets of a subdivision of $K_{3,3}$ in $G$, which is a contradiction.

Thus $\Gamma$ must contain a vertex which is not in $Z_1$, so without loss of generality it contains $u_1$. If it doesn't contain $u_2$, then the rest of its vertices are in $Z_1$, and the two which are not adjacent to $u_1$, say $p_1$ and $p_2$, must be adjacent to $u_2$. Let $q_1$ and $q_2$ be the other two vertices in $\Gamma$. Then $\{u_1,w,p_1\}$ and $\{u_2,q_1,q_2\}$ form the partite sets of a subdivision of $K_{3,3}$ in $H$, giving a contradiction. If $\Gamma$ does contain $u_2$, then since $u_1$ and $u_2$ are neighbours, the other three vertices in $\Gamma$ are in $Z_1$. One of these three must be adjacent to neither $u_1$ nor $u_2$, but no such vertex exists. Hence $l_{\{u_1,w\}}=0$. Similarly $l_{\{u_2,w\}}=0$. If $l_{\{u_1,u_2,w\}}=0$ then $F_1$ is of type 1. It is straightforward to use the planarity of $H$ to show that $l_{\{u_1,u_2,w\}}\leq 1$, so if $l_{\{u_1,u_2,w\}}\neq 0$, then $F_1$ is of type 2.

If condition (ii) doesn't hold, then neither does condition (ii)', so conditions (i) and (i)' both hold. In particular, 
\begin{enumerate}[label=(\alph*)]
    \item there is a vertex in $L_{\{u_1,w\}}$ which has a neighbour in $L_{\{u_1,u_2\}}$,
    \item there is a vertex in $L_{\{u_2,w\}}$ which has a neighbour in $L_{\{u_1,u_2\}}$,
    \item every vertex in $L_{\{u_1,u_2\}}$ has a neighbour in $L_{\{w\}}\cup L_{\{u_1,w\}}$, and
    \item every vertex in $L_{\{w\}}$ has a neighbour in $L_{\{u_1,u_2\}}$.
\end{enumerate}

Condition (a) implies that $l_{\{u_1,u_2\}}\geq 1$, and it is straightforward to deduce from (c) and the planarity of $H$ that $l_{\{u_1,u_2\}}\leq 1$. It is also straightforward to use (a) and the planarity of $H$ to deduce that $l_{\{u_1,u_2,w\}}=0$. It remains to determine $l_{\{u_1,w\}}$, $l_{\{u_2,w\}}$ and $l_{\{w\}}$.

Let $z_1$ be the unique vertex in $L_{\{u_1,u_2\}}$ and let $z_2$ be a vertex in $L_{\{u_1,w\}}$ which is adjacent to $z_1$. Let $m=|L_{\{w\}}|+3$ and let $z_m$ be a vertex in $L_{\{u_2,w\}}$ which is adjacent to $z_1$. Consider $H[\{u_1,u_2,w,a_{|A|}, b_1,z_1,z_2,z_m\}]$ with the edge $z_2z_m$ deleted if present, and the drawing of this induced by our drawing of $H$. Face $F_1$ is split into six faces whose boundaries are the cycles in a fixed list. Let $E_1$, $E_2$, and $E_3$ respectively be the faces whose boundaries are formed of the cycles $u_1a_{|A|}wz_2$, $u_2b_1wz_m$, and $z_1z_2wz_m$ (see Figure \ref{Proving type 3}). Any vertices in $L_{\{u_1,w\}}\setminus\{z_2\}$ must be in $E_1$ and any vertices in $L_{\{u_2,w\}}\setminus\{z_m\}$ must be in $E_2$. By (d), every vertex in $L_{\{w\}}$ is in $E_3$.

\begin{figure}[!ht]
    \centering
    \captionsetup{justification=centering}
    \includegraphics[width=0.5\textwidth]{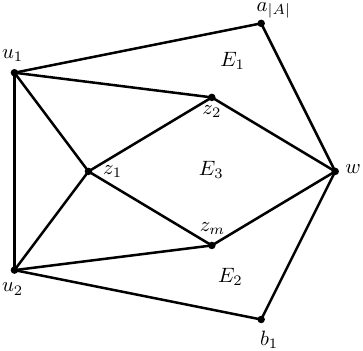}
    \caption[width=0.8\textwidth]{Illustration of the drawing of $H[\{u_1,u_2,w,a_{|A|}, b_1,z_1,z_2,z_m\}]$ with the edge $z_2z_m$ deleted if present}
    \label{Proving type 3}
\end{figure}

Suppose there exists some $z\in L_{\{u_1,w\}}\setminus\{z_2\}$. Then since $z\not\in A$, $z$ has a principal neighbour other than $u_1$ and $w$, say $z'$. Then $z'$ is either $z_2$ or is in the interior of $E_1$, so in particular $z'\in  L_{\{u_1,w\}}$. Let $\Gamma$ be an induced 5-cycle in $G$ containing $z$ and $z'$. Then $\Gamma$ does not contain $u_1$ or $w$ since these are common neighbours of $z$ and $z'$. Hence every vertex in $\Gamma$ is either $z_2$ or is in the interior of $E_1$, so they are all adjacent to $u_1$ and $w$. By an argument similar to that used in the proof of Lemma \ref{X and Y}, this implies the existence of a subdivision of $K_{3,3}$ in $H$, which is a contradiction. Therefore $L_{\{u_1,w\}}=\{z_2\}$ and similarly $L_{\{u_2,w\}}=\{z_m\}$.

Consider adding the vertices in $L_{\{w\}}$ and their edges to $w$ and $z_1$ to the drawing of $H[\{u_1,u_2,w,a_{|A|}, b_1,z_1,z_2,z_m\}]$. We see that we can label them as $z_3,\dots,z_{m-1}$ such that to preserve planarity the only other edges which could be present in $F_1$ are of the form $z_iz_{i+1}$ for $2\leq i\leq m-1$. Hence $F_1$ is of type 3. This completes the proof of the claim.
\end{proof}

To summarise, we know that:
\begin{itemize}
    \item there are exactly $|A||B|+|A||C|+|B||C|$ induced 5-cycles in $H-Z$, 
    \item $a$, $b$, and $c$ are in $|B|+|C|$, $|A|+|C|$, and $|A|+|B|$ induced 5-cycles in $H-Z$ respectively,
    \item each of $a$, $b$, and $c$ is in exactly $\frac{2n-8}{3}$ induced 5-cycles in $H$,
    \item every vertex of $Z$ is in $F_1$, $F_2$, or $F_3$,
    \item each of $F_1$, $F_2$, and $F_3$ is of type 1, 2, or 3,
    \item there are no induced 5-cycles in $H$ containing two vertices in $A\cup B\cup C$ and a vertex in $Z$, and
    \item there are no induced 5-cycles in $H$ containing one of $a$, $b$, and $c$ and vertices in $Z$ from two of $F_1$, $F_2$, and $F_3$.
\end{itemize}

For each assignment of types 1, 2, and 3 to faces $F_1$, $F_2$, and $F_3$ and for each $x\in \{a,b,c\}$, we now count how many induced 5-cycles there are in $H$ containing $x$ and a vertex in $Z$. After that, we will count how many induced 5-cycles there are in $H[\{u_1,u_2,u_3,w\}\cup Z]$ for each of these assignments. We will then use the second and third points above to calculate the sizes of $A$, $B$, and $C$ in each case. We can then use all of this information to determine the total number of induced 5-cycles in $H$ in each case.

For the first of these steps it is sufficient by symmetry to count the induced 5-cycles containing $x$ and a vertex in $F_1$ for each $x\in \{a,b,c\}$ and each assignment of a type to $F_1$. If $F_1$ is of type 1, then there are clearly no induced 5-cycles containing $a$ and a vertex in $F_1$, and similarly for $b$ and $c$. If $F_1$ is of type 2, containing a single vertex $z$, then $u_1zwcu_3$ and $u_2zwcu_3$ are induced 5-cycles in $H$. We know there are at most twice as many induced 5-cycles in $H$ containing one of $a$, $b$, and $c$ and a vertex in $F_1$ as there are vertices in $F_1$, so there are exactly two induced 5-cycles containing $c$ and $z$, and none containing $a$ and $z$ or $b$ and $z$.

Finally, if $F_1$ is of type 3 with $m\geq 3$ vertices labelled as in the definition of a type 3 face (and Figure \ref{Z type 3}), then $H$ contains the induced 5-cycles $u_1z_2wbu_2$, $u_1z_2wcu_3$, $u_2z_mwau_1$, $u_2z_mwcu_3$, $u_1z_1z_iwa$ for $3\leq i\leq m$, and $u_2z_1z_iwb$ for $2\leq i\leq m-1$. We have identified $2m$ suitable induced 5-cycles, so as above we know this is all of them. So $a$ and $b$ are each in exactly $m-1$ induced 5-cycles in $H$ containing a vertex in $F_1$, and $c$ is exactly in two.

We now determine the number of induced 5-cycles in $H[\{u_1,u_2,u_3,w\}\cup Z]$ in each case. Note that every induced 5-cycle in this graph contains a vertex of $Z$, and no induced 5-cycle can contain a vertex of $Z$ from each of $F_1$, $F_2$, and $F_3$. Let $Z_1\subseteq Z$ be the set of vertices in $F_1$ as before, and let $Z_2\subseteq Z$ be the set of vertices in $F_2$. Then by symmetry it is sufficient to calculate the number of induced 5-cycles in $H[\{u_1,u_2,u_3,w\}\cup Z_1]$ in each case, and the number of induced 5-cycles in $H[\{u_1,u_2,u_3,w\}\cup Z_1\cup Z_2]$ which contain vertices from both $Z_1$ and $Z_2$ in each case.

We first count the number of induced 5-cycles in $H[\{u_1,u_2,u_3,w\}\cup Z_1]$. Clearly there are none if $F_1$ is of type 1 or 2. Suppose $F_1$ is of type 3 with $m\geq 3$ vertices labelled as in the definition of a face of type 3. The graph $H[\{u_1,u_2,u_3,w\}\cup Z_1]$ is illustrated in Figure \ref{Counting cycles type 3}, where a red line indicates an edge which may or may not be present.

\begin{figure}
    \centering
    \captionsetup{justification=centering}
    \begin{subfigure}[b]{0.48\textwidth}
        \centering
        \captionsetup{justification=centering}
        \includegraphics[width=\textwidth]{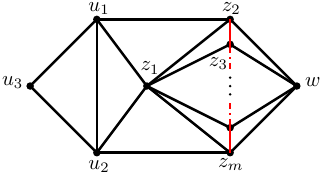}
        \caption{$H[\{u_1,u_2,u_3,w\}\cup Z_1]$ if $F_1$ is of type 3}
        \label{Counting cycles type 3}
        \vspace*{1cm}
     \end{subfigure}
     \hfill
     \begin{subfigure}[b]{0.48\textwidth}
        \centering
        \captionsetup{justification=centering}
        \includegraphics[width=0.5\textwidth]{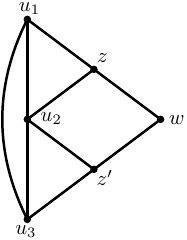}
        \caption{$H[\{u_1,u_2,u_3,w\}\cup Z_1\cup Z_2]$ if $F_1$ and $F_2$ are both of type 2}
        \label{Counting cycles type 2 and 2}
        \vspace*{1cm}
     \end{subfigure}
     
     \begin{subfigure}[b]{0.48\textwidth}
        \centering
        \captionsetup{justification=centering}
        \includegraphics[width=\textwidth]{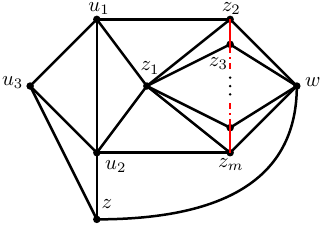}
        \vspace*{7mm}
        \caption{$H[\{u_1,u_2,u_3,w\}\cup Z_1\cup Z_2]$ if $F_1$ is of type 3 and $F_2$ is of type 2}
        \label{Counting cycles type 2 and 3}
     \end{subfigure}
    \hfill
     \begin{subfigure}[b]{0.48\textwidth}
        \centering
        \captionsetup{justification=centering}
        \includegraphics[width=\textwidth]{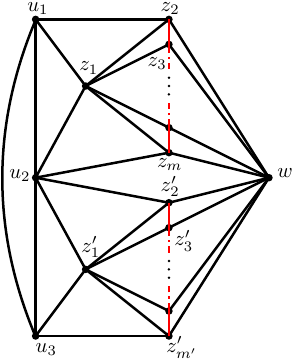}
        \caption{$H[\{u_1,u_2,u_3,w\}\cup Z_1\cup Z_2]$ if $F_1$ and $F_2$ are both of type 3}
        \label{Counting cycles type 3 and 3}
     \end{subfigure}
        \caption{Certain subgraphs of $H$}
        \label{Counting}
\end{figure}

In this graph $u_3$ is only adjacent to $u_1$ and $u_2$ and these are themselves neighbours, so $u_3$ is not in an induced 5-cycle in this graph. Next, $z_1$ is adjacent to all the remaining vertices except $w$, so it also cannot be in an induced 5-cycle in this graph. If an induced 5-cycle contains $w$, then it contains exactly two of the vertices $z_2,\dots,z_m$, and these two must not be neighbours. The remaining two vertices must be $u_1$ and $u_2$, and hence the only possible induced 5-cycle containing $w$ is $u_1z_2wz_mu_2$. This induced 5-cycle is realised if and only if $m\geq 4$, or $m=3$ and $z_2$ and $z_3$ are not neighbours.

Now suppose $\Gamma$ is an induced 5-cycle which does not contain $w$. Then it must contain at least 3 vertices from $z_2,\dots,z_m$. Clearly it cannot only contain vertices from $z_2,\dots,z_m$, so it contains at least one of $u_1$ and $u_2$. Each of these only has one neighbour in $\{z_2,\dots,z_m\}$, so in fact $\Gamma$ contains both $u_1$ and $u_2$. Hence it also contains $z_2$ and $z_m$, and the remaining vertex is a common neighbour of these two among $z_3,\dots,z_{m-1}$. So for such a cycle to appear we must have $m=4$, and the edges $z_2z_3$ and $z_3z_4$ must be present. This condition is also sufficient for the induced cycle $u_1z_2z_3z_4u_2$ to appear.

For each assignment of types 1, 2, and 3 to faces $F_1$ and $F_2$, we now count the number of induced 5-cycles in $H[\{u_1,u_2,u_3,w\}\cup Z_1\cup Z_2]$ which contain vertices from both $Z_1$ and $Z_2$. If either face is of type 1, then clearly there are no such cycles. If both faces are of type 2, then label the vertex in $F_1$ as $z$ and the vertex in $F_2$ as $z'$. The graph $H[\{u_1,u_2,u_3,w\}\cup Z_1\cup Z_2]$ is illustrated in Figure \ref{Counting cycles type 2 and 2}. It is straightforward to see (for example by considering whether or not the cycle contains $w$) that there is exactly one induced 5-cycle of the required form, namely $u_1zwz'u_3$.

Now suppose that $F_1$ is of type 3 with $m\geq 3$ vertices and $F_2$ is of type 2. Label the vertices in $F_1$ as $z_1,\dots,z_m$ in the usual way, and the vertex in $F_2$ as $z$. The graph $H[\{u_1,u_2,u_3,w\}\cup Z_1\cup Z_2]$ is illustrated in Figure \ref{Counting cycles type 2 and 3}, where a red line indicates an edge which may or may not be present.

Let $\Gamma$ be a 5-cycle in the graph containing $z$ and one of $z_1,\dots,z_m$. Then since $\Gamma$ contains $z$ it must also contain $w$. Suppose $\Gamma$ contains $z_1$, then since $z_1$ is not adjacent to $w$ the cycle must also contain a common neighbour of $z_1$ and $w$, i.e. one of $z_2,\dots,z_m$. So $\Gamma$ contains the path $zwz_iz_1$ for some $2\leq i\leq m$. The only common neighbour of $z$ and $z_1$ is $u_2$, so $\Gamma$ is $zwz_iz_1u_2$ for some $2\leq i\leq m$. This induced 5-cycle is realised if and only if $i\neq m$, so there are exactly $m-2$ such 5-cycles.

Now suppose $\Gamma$ does not contain $z_1$. Then it must contain one of $z_2,\dots,z_m$, and since these are all neighbours of $w$ it must in fact contain exactly one of them. For $3\leq i\leq m-1$, $z_i$ has no neighbours outside $\{w,z_1,\dots,z_m\}$, so $\Gamma$ must contain $z_2$ or $z_m$. If it contains $z_m$, then it also contains $u_2$, but this is a neighbour of $z$. Hence $\Gamma$ contains the path $zwz_2u_1$. The only common neighbours of $u_1$ and $z$ are $u_2$ and $u_3$, and the induced 5-cycle is realised in both cases. So in total there are $m$ induced 5-cycles in the graph containing a vertex from $Z_1$ and a vertex from $Z_2$.

Finally, suppose $F_1$ and $F_2$ are both of type 3, containing $m\geq 3$ and $m'\geq 3$ vertices respectively. Label the vertices in $F_1$ as $z_1,\dots,z_m$ in the usual way, and the vertices in $F_2$ as $z'_1,\dots,z'_{m'}$ in the analogous way, as illustrated in Figure \ref{Counting cycles type 3 and 3}. As usual, red lines in the figure indicate edges which may or may not be present.

Let $\Gamma$ be a 5-cycle in the graph containing a vertex from $Z_1$ and a vertex from $Z_2$. Since $\Gamma$ contains vertices in both $F_1$ and $F_2$, it must contain two vertices on the boundaries of each. Therefore it contains a vertex which is on the boundary of both faces, i.e. $w$ or $u_2$. If it does not contain $w$, then in order to contain two vertices on the boundary of each face it must contain $u_1$, $u_2$, and $u_3$ which is impossible. Hence $\Gamma$ contains $w$.

Suppose $\Gamma$ contains $z_1$. If it does not also contain one of $z_2,\dots,z_m$, then it contains $u_1$ and $u_2$ as these are the only other vertices adjacent to $z_1$. But then $\Gamma$ contains $u_1$, $u_2$, and $z_1$, which form a triangle, giving a contradiction. So $\Gamma$ contains one of $z_2,\dots,z_m$. It cannot contain more than one of $z_2,\dots,z_m$ since these are all common neighbours of $z_1$ and $w$.

Hence $\Gamma$ contains a path $u_iz_1z_jw$ for some $i\in\{1,2\}$ and $j\in\{2,\dots,m\}$. There are no common neighbours of $u_1$ and $w$ which are not adjacent to $z_1$, and the only common neighbour of $u_2$ and $w$ which is not adjacent to $z_1$ is $z'_2$. Hence if $\Gamma$ contains $z_1$, then it is of the form $u_2z_1z_jwz'_2$ for some $j\in\{2,\dots,m\}$. This induced 5-cycle is realised if and only if $j\in\{2,\dots,m-1\}$, so there are exactly $m-2$ induced cycles of the required form containing $z_1$. Similarly there are exactly $m'-2$ containing $z'_1$.

Now suppose $\Gamma$ contains neither $z_1$ nor $z'_1$. Then it contains one vertex from $z_2,\dots,z_m$ and one from $z'_2,\dots,z'_{m'}$. All of these vertices are neighbours of $w$, so in fact $\Gamma$ must contain exactly one from each list. The remaining two vertices in $\Gamma$ must be picked from $u_1$, $u_2$, and $u_3$. None of $z_3,\dots,z_{m-1}$ or $z'_3,\dots,z'_{m'-1}$ are adjacent to any of these vertices, so the neighbours of $w$ in the cycle must be one of $z_2$ and $z_m$ and one of $z'_2$ and $z'_{m'}$. For each such combination except $z_m$ and $z'_2$, we see there is a single induced 5-cycle of the required form, and for $z_m$ and $z'_2$ there are none. So in total there are $m+m'-1$ induced 5-cycles in the graph containing a vertex in $Z_1$ and a vertex in $Z_2$.

To summarise our findings:
\begin{itemize}
    \item if $F_1$ is of type 1, then there are no induced 5-cycles containing a vertex in $F_1$,
    \item if $F_1$ is of type 2, then there are no induced 5-cycles containing $a$ or $b$ and a vertex in $F_1$, and there are exactly two containing $c$ and a vertex in $F_1$,
    \item if $F_1$ is of type 3 with $m\geq 3$ vertices, then for $x\in\{a,b\}$ there are exactly $m-1$ induced 5-cycles containing $x$ and a vertex in $F_1$, and there are exactly two containing $c$ and a vertex in $F_1$,
    \item $H[\{u_1,u_2,u_3,w\}\cup Z]$ contains exactly
    \begin{itemize}
        \item one induced 5-cycle per face of type 3 with $m\geq 5$ vertices, four vertices but not both optional edges present, or three vertices with the optional edge not present,
        \item two induced 5-cycles per face of type 3 with four vertices and both optional edges present,
        \item one induced 5-cycle per pair of faces of type 2,
        \item $m$ induced 5-cycles per pair of faces where one is of type 2 and the other is of type 3 with $m$ vertices,
        \item $m$+$m'$-1 induced 5-cycles per pair of faces of type 3, where one has $m$ vertices and the other has $m'$,
    \end{itemize}
    and no further induced 5-cycles.
\end{itemize}

For each assignment of types to $F_1$, $F_2$, and $F_3$, we can now calculate the sizes of $A$, $B$, and $C$ and thence can determine the total number of induced 5-cycles in $H$ in each case. The details of this process are laid out in Appendix \ref{appendix}, where the following claim is proved.

\begin{claim}\label{appendix claim}
$H$ is the principal graph of the required form on $n$ vertices.
\end{claim}

Label the vertices of $H$ according to Definition \ref{reqd form def}. Fix a drawing of $G$, then by Lemma \ref{drawing of required form} we may assume that the boundaries of the faces in the induced drawing of $H$ consist of the cycles listed in the statement of that lemma. Consider adding the edges in $E(G)\setminus E(H)$ back to this drawing of $H$. Clearly no edge is added to a triangular face. The remaining faces' boundaries are cycles of the form $u_iywy'$ for some $i\in \{1,2,3\}$ and $y,y'\in A\cup B\cup C\cup Z$. Hence every edge in $E(G)\setminus E(H)$ is of the form $u_iw$ for some $i\in \{1,2,3\}$, or $yy'$ for some $y,y'\in A\cup B\cup C\cup Z$ such that there exists $i\in \{1,2,3\}$ such that cycle $u_iywy'$ appears in the list in Lemma \ref{drawing of required form}. The sets of induced 5-cycles in $G$ and $H$ are the same, and there is an induced 5-cycle in $H$ containing $u_i$ and $w$ for each $i\in \{1,2,3\}$, so $u_iw\not\in E(G)\setminus E(H)$ for all $i\in \{1,2,3\}$. Analysing the list of cycles in Lemma \ref{drawing of required form}, we see that $G$ is of the required form.

By Lemma \ref{5-cycle count}, $H$ contains exactly $\frac{1}{3}(n^2-8n+22)$ induced 5-cycles. We know that $G$ has the same set of induced 5-cycles as $H$, so $G$ contains exactly $\frac{1}{3}(n^2-8n+22)$ induced 5-cycles too.
\end{proof}

\section*{Acknowledgements}
Thank you to Tom Johnston, Emil Powierski, and Jane Tan for helpful discussions. Particular thanks to Tom for finding the graph in Figure \ref{10 vertex}, and to Emil and Jane for spotting that an argument to show $f_I(n,C_4)=\frac{1}{2}(n^2-5n+6)$ for large $n$ could be extended to show that $K_{2,n-2}$ uniquely achieves the maximum. Thank you also to Alex Scott for support and helpful comments. The author is grateful to the Heilbronn Institute for Mathematical Research for their support.

\appendix
\section{Proof of Claim \ref{appendix claim}}\label{appendix}
Let $K_1=\frac{n}{3}-|A|$, $K_2=\frac{n}{3}-|B|$, and $K_3=\frac{n}{3}-|C|$. For each assignment of types to $F_1$, $F_2$, and $F_3$ (ignoring symmetric cases), we will assume that $H$ follows this assignment and go through the following process. We know that each of $a$, $b$, and $c$ is in exactly $\frac{2n-8}{3}$ induced 5-cycles in $H$. We also know that $a$ is in exactly $|B|+|C|$ induced 5-cycles which avoid $Z$, and for each assignment of types to faces we know how many induced 5-cycles there are which contain $a$ and a vertex of $Z$. Repeating for $b$ and $c$ gives three linear equations which we can solve to find $K_1$, $K_2$, and $K_3$.

Once we know the sizes of $A$, $B$, and $C$, we can use our earlier findings to determine the total number of induced 5-cycles in $H$ in each case. There are $|A||B|+|A||C|+|B||C|$ induced 5-cycles avoiding $Z$, $|A|\left(\frac{2n-8}{3}-|B|-|C|\right) + |B|\left(\frac{2n-8}{3}-|A|-|C|\right) + |C|\left(\frac{2n-8}{3}-|A|-|B|\right)$ induced 5-cycles containing a vertex of $A\cup B\cup C$ and a vertex of $Z$, and in each case we know how many induced 5-cycles there are avoiding $A\cup B \cup C$. Note that by our earlier findings, for a face of type 3 with three vertices the number of induced 5-cycles in $H$ can only increase if the optional edge is not present, so we may assume this is always the case. Similarly, for a face of type 3 with four vertices we may assume that both optional edges are present.

Of the below cases, case 7 with $m=4$ uniquely gives the most induced 5-cycles, so since $H$ contains $f_I(n,C_5)$ induced 5-cycles we have $f_I(n,C_5)\leq \frac{1}{3}(n^2-8n+22)$ for this $n$. By Lemma \ref{5-cycle count}, in fact $f_I(n,C_5)=\frac{1}{3}(n^2-8n+22)$ for this $n$, and hence case 7 with $m=4$ must occur. In other words, $H$ is the principal graph of the required form on $n$ vertices.

\begin{caseof}
    \case{$F_1$, $F_2$, and $F_3$ type 1.}{Vertices $a$, $b$, and $c$ are in $\frac{2n}{3}-K_2-K_3$, $\frac{2n}{3}-K_1-K_3$, and $\frac{2n}{3}-K_1-K_2$ induced 5-cycles respectively. Solving, we obtain $K_1=K_2=K_3=\frac{4}{3}$. There are no induced 5-cycles avoiding $A\cup B\cup C$. Hence there are $3(\frac{n-4}{3})^2=\frac{1}{3}(n^2-8n+16)$ induced 5-cycles in $H$.}
    
    \case{$F_1$ type 2, $F_2$ and $F_3$ type 1.}{Vertices $a$, $b$, and $c$ are in $\frac{2n}{3}-K_2-K_3$, $\frac{2n}{3}-K_1-K_3$, and $\frac{2n}{3}-K_1-K_2+2$ induced 5-cycles respectively. Solving, we obtain $K_1=K_2=\frac{7}{3}$ and $K_3=\frac{1}{3}$. There are no induced 5-cycles avoiding $A\cup B\cup C$. Hence there are $2(\frac{n-7}{3})(\frac{n-1}{3})+(\frac{n-7}{3})^2+2(\frac{n-1}{3})=\frac{1}{3}(n^2-8n+19)$ induced 5-cycles in $H$.}
    
    \case{$F_1$ and $F_2$ type 2, $F_3$ type 1.}{Vertices $a$, $b$, and $c$ are in $\frac{2n}{3}-K_2-K_3+2$, $\frac{2n}{3}-K_1-K_3$, and $\frac{2n}{3}-K_1-K_2+2$ induced 5-cycles respectively. Solving, we obtain $K_1=K_3=\frac{4}{3}$ and $K_2=\frac{10}{3}$. There is exactly one induced 5-cycle avoiding $A\cup B\cup C$. Hence there are $2(\frac{n-4}{3})(\frac{n-10}{3})+(\frac{n-4}{3})^2+4(\frac{n-4}{3})+1=\frac{1}{3}(n^2-8n+19)$ induced 5-cycles in $H$.}
    
    \case{$F_1$, $F_2$, and $F_3$ type 2.}{Vertices $a$, $b$, and $c$ are in $\frac{2n}{3}-K_2-K_3+2$, $\frac{2n}{3}-K_1-K_3+2$, and $\frac{2n}{3}-K_1-K_2+2$ induced 5-cycles respectively. Solving, we obtain $K_1=K_2=K_3=\frac{7}{3}$. There are exactly 3 induced 5-cycles avoiding $A\cup B\cup C$. Hence there are $3(\frac{n-7}{3})^2+6(\frac{n-7}{3})+3=\frac{1}{3}(n^2-8n+16)$ induced 5-cycles in $H$.}
    
    \case{$F_1$ type 3 with $m$ vertices, $F_2$ and $F_3$ type 1.}{Vertices $a$, $b$, and $c$ are in $\frac{2n}{3}-K_2-K_3+m-1$, $\frac{2n}{3}-K_1-K_3+m-1$, and $\frac{2n}{3}-K_1-K_2+2$ induced 5-cycles respectively. Solving, we obtain $K_1=K_2=\frac{7}{3}$ and $K_3=m-\frac{2}{3}$. There is exactly one induced 5-cycle avoiding $A\cup B\cup C$ for $m\neq 4$. Hence there are $2(\frac{n-7}{3})(\frac{n+2}{3}-m)+(\frac{n-7}{3})^2+2(m-1)(\frac{n-7}{3})+2(\frac{n+2}{3}-m)+1=\frac{1}{3}(n^2-8n+28-6m)$ induced 5-cycles in $H$ for $m\neq 4$. Since $m\geq 3$, this is at most $\frac{1}{3}(n^2-8n+10)$. For $m=4$ there is one more induced 5-cycle, so there are $\frac{1}{3}(n^2-8n+7)$ in total.}
    
    \case{$F_1$ type 3 with $m$ vertices, $F_2$ type 2, $F_3$ type 1.}{Vertices $a$, $b$, and $c$ are in $\frac{2n}{3}-K_2-K_3+m+1$, $\frac{2n}{3}-K_1-K_3+m-1$, and $\frac{2n}{3}-K_1-K_2+2$ induced 5-cycles respectively. Solving, we obtain $K_1=\frac{4}{3}$, $K_2=\frac{10}{3}$, and $K_3=m+\frac{1}{3}$. There are exactly $m+1$ induced 5-cycles avoiding $A\cup B\cup C$ for $m\neq 4$. Hence there are $(\frac{n-4}{3})(\frac{n-10}{3})+(\frac{n-4}{3})(\frac{n-1}{3}-m)+(\frac{n-10}{3})(\frac{n-1}{3}-m)+(m+1)(\frac{n-4}{3})+(m-1)(\frac{n-10}{3})+2(\frac{n-1}{3}-m)+m+1=\frac{1}{3}(n^2-8n+25-3m)$ induced 5-cycles in $H$ for $m\neq 4$. Since $m\geq 3$, this is at most $\frac{1}{3}(n^2-8n+16)$. For $m=4$ there is one more induced 5-cycle, so there are $\frac{1}{3}(n^2-8n+16)$ in total.}
    
    \case{$F_1$ type 3 with $m$ vertices, $F_2$ and $F_3$ type 2.}{Vertices $a$, $b$, and $c$ are in $\frac{2n}{3}-K_2-K_3+m+1$, $\frac{2n}{3}-K_1-K_3+m+1$, and $\frac{2n}{3}-K_1-K_2+2$ induced 5-cycles respectively. Solving, we obtain $K_1=K_2=\frac{7}{3}$ and $K_3=m+\frac{4}{3}$. There are exactly $2m+2$ induced 5-cycles avoiding $A\cup B\cup C$ for $m\neq 4$. Hence there are $2(\frac{n-7}{3})(\frac{n-4}{3}-m)+(\frac{n-7}{3})^2+2(m+1)(\frac{n-7}{3})+2(\frac{n-4}{3}-m)+2m+2=\frac{1}{3}(n^2-8n+19)$ induced 5-cycles in $H$ for $m\neq 4$. For $m=4$ there is one more, so there are $\frac{1}{3}(n^2-8n+22)$.}
    
    \case{$F_1$ type 3 with $m$ vertices, $F_2$ type 3 with $m'$ vertices, $F_3$ type 1.}{Vertices $a$, $b$, and $c$ are in $\frac{2n}{3}-K_2-K_3+m+1$, $\frac{2n}{3}-K_1-K_3+m+m'-2$, and $\frac{2n}{3}-K_1-K_2+m'+1$ induced 5-cycles respectively. Solving, we obtain $K_1=m'+\frac{1}{3}$, $K_2=\frac{10}{3}$, and $K_3=m+\frac{1}{3}$. There are exactly $m+m'+1$ induced 5-cycles avoiding $A\cup B\cup C$ if $m,m'\neq 4$. Hence there are $(\frac{n-1}{3}-m')(\frac{n-10}{3})+(\frac{n-1}{3}-m')(\frac{n-1}{3}-m)+(\frac{n-10}{3})(\frac{n-1}{3}-m)+(m+1)(\frac{n-1}{3}-m')+(m+m'-2)(\frac{n-10}{3})+(m'+1)(\frac{n-1}{3}-m)+m+m'+1=\frac{1}{3}(n^2-8n+28-3mm')$ induced 5-cycles in $H$ for $m, m'\neq 4$. Since $m,m'\geq 3$, this is at most $\frac{1}{3}(n^2-8n+1)$. If one or both of $m$ and $m'$ is 4, then this increases by at most 2, to give a total of at most $\frac{1}{3}(n^2-8n+7)$.}
    
    \case{$F_1$ type 3 with $m$ vertices, $F_2$ type 3 with $m'$ vertices, $F_3$ type 2.}{Vertices $a$, $b$, and $c$ are in $\frac{2n}{3}-K_2-K_3+m+1$, $\frac{2n}{3}-K_1-K_3+m+m'$, and $\frac{2n}{3}-K_1-K_2+m'+1$ induced 5-cycles respectively. Solving, we obtain $K_1=m'+\frac{4}{3}$, $K_2=\frac{7}{3}$, and $K_3=m+\frac{4}{3}$. There are exactly $2m+2m'+1$ induced 5-cycles avoiding $A\cup B\cup C$ if $m,m'\neq 4$. Hence there are $(\frac{n-4}{3}-m')(\frac{n-7}{3})+(\frac{n-4}{3}-m')(\frac{n-4}{3}-m)+(\frac{n-7}{3})(\frac{n-4}{3}-m)+(m+1)(\frac{n-4}{3}-m')+(m+m')(\frac{n-7}{3})+(m'+1)(\frac{n-4}{3}-m)+2m+2m'+1=\frac{1}{3}(n^2-8n+19+3(m+m'-mm'))$ induced 5-cycles in $H$ for $m,m'\neq 4$. For $m,m'\geq 3$ we see that $m+m'-mm'$ is decreasing in $m$ and $m'$, so this is at most $\frac{1}{3}(n^2-8n+10)$. If one or both of $m$ and $m'$ is 4, then this increases by at most 2, to give a total of at most $\frac{1}{3}(n^2-8n+16)$.}
    
    \case{$F_1$, $F_2$, and $F_3$ type 3 with $m$, $m'$, and $m''$ vertices respectively.}{Vertices $a$, $b$, and $c$ are in $\frac{2n}{3}-K_2-K_3+m+m''$, $\frac{2n}{3}-K_1-K_3+m+m'$, and $\frac{2n}{3}-K_1-K_2+m'+m''$, induced 5-cycles respectively. Solving, we obtain $K_1=m'+\frac{4}{3}$, $K_2=m''+\frac{4}{3}$, and $K_3=m+\frac{4}{3}$. There are exactly $2m+2m'+2m''$ induced 5-cycles avoiding $A\cup B\cup C$ if $m,m',m''\neq 4$. Hence there are $(\frac{n-4}{3}-m')(\frac{n-4}{3}-m'')+(\frac{n-4}{3}-m')(\frac{n-4}{3}-m)+(\frac{n-4}{3}-m'')(\frac{n-4}{3}-m)+(m+m'')(\frac{n-4}{3}-m')+(m+m')(\frac{n-4}{3}-m'')+(m'+m'')(\frac{n-4}{3}-m)+2m+2m'+2m''=\frac{1}{3}(n^2-8n+16+3(2m+2m'+2m''-mm'-m'm''-mm''))$ induced 5-cycles in $H$ for $m,m',m''\neq 4$. For $m,m',m''\geq 3$ we see that $2m+2m'+2m''-mm'-m'm''-mm''$ is decreasing in $m$, $m'$, and $m''$, so this is at most $\frac{1}{3}(n^2-8n-11)$. If any of $m$, $m'$, and $m''$ are 4, then this increases by at most 3, to give a total of at most $\frac{1}{3}(n^2-8n-2)$.}
\end{caseof}

\begin{thebibliography}{10}

\bibitem{alon}
N.~Alon and Y.~Caro.
\newblock On the number of subgraphs of prescribed type of planar graphs with a
  given number of vertices.
\newblock In M.~Rosenfeld and J.~Zaks, editors, {\em Annals of Discrete
  Mathematics (20): Convexity and Graph Theory}, volume~87 of {\em
  North-Holland Mathematics Studies}, pages 25--36. North-Holland, 1984.

\bibitem{balogh}
J.~Balogh, P.~Hu, B.~Lidický and F.~Pfender.
\newblock Maximum density of induced 5-cycle is achieved by an iterated blow-up
  of 5-cycle.
\newblock {\em European Journal of Combinatorics}, {\bfseries 52}:47--58, 2016.

\bibitem{cox2021counting}
C.~Cox and R.~R. Martin.
\newblock Counting paths, cycles and blow-ups in planar graphs.
  \textit{arXiv:2101.05911} preprint, 2021.

\bibitem{cox2021maximum}
C.~Cox and R.~R. Martin.
\newblock The maximum number of 10- and 12-cycles in a planar graph.
  \textit{arXiv:2106.02966} preprint, 2021.

\bibitem{eppstein}
D.~Eppstein.
\newblock Connectivity, graph minors, and subgraph multiplicity.
\newblock {\em Journal of Graph Theory}, {\bfseries 17}(3):409--416, 1993.

\bibitem{even-zohal}
C.~Even-Zohar and N.~Linial.
\newblock A note on the inducibility of 4-vertex graphs.
\newblock {\em Graphs and Combinatorics}, {\bfseries 31}:1367--1380, 2015.

\bibitem{ghosh2020maximum}
D.~Ghosh, E.~Győri, O.~Janzer, A.~Paulos, N.~Salia and O.~Zamora.
\newblock The maximum number of induced {$C_5$}'s in a planar graph.
  \textit{arXiv:2004.01162v1} preprint, 2020.

\bibitem{ghosh2021maximum}
D.~Ghosh, E.~Győri, O.~Janzer, A.~Paulos, N.~Salia and O.~Zamora.
\newblock The maximum number of induced {$C_5$}'s in a planar graph.
  \textit{arXiv:2004.01162v2} preprint, 2021.

\bibitem{ghosh2021maximumv3}
D.~Ghosh, E.~Győri, O.~Janzer, A.~Paulos, N.~Salia and O.~Zamora.
\newblock The maximum number of induced {$C_5$}'s in a planar graph.
  \textit{arXiv:2004.01162v3} preprint, 2021.

\bibitem{ghoshp5}
D.~Ghosh, E.~Győri, R.~R. Martin, A.~Paulos, N.~Salia, C.~Xiao and O.~Zamora.
\newblock The maximum number of paths of length four in a planar graph.
\newblock {\em Discrete Mathematics}, {\bfseries 344}(5):112317, 2021.

\bibitem{grzesik}
A.~Grzesik, E.~Győri, A.~Paulos, N.~Salia, C.~Tompkins and O.~Zamora.
\newblock The maximum number of paths of length three in a planar graph.
  \textit{arXiv:1909.13539} preprint, 2021.

\bibitem{gyori}
E.~Győri, A.~Paulos, N.~Salia, C.~Tompkins and O.~Zamora.
\newblock The maximum number of pentagons in a planar graph.
  \textit{arXiv:1909.13532} preprint, 2019.

\bibitem{gyori2020generalized}
E.~Győri, A.~Paulos, N.~Salia, C.~Tompkins and O.~Zamora.
\newblock Generalized planar {T}ur\'an numbers. \textit{arXiv:2002.04579}
  preprint, 2020.

\bibitem{hakimi}
S.~L. Hakimi and E.~F. Schmeichel.
\newblock On the number of cycles of length k in a maximal planar graph.
\newblock {\em Journal of Graph Theory}, {\bfseries 3}(1):69--86, 1979.

\bibitem{hatami}
H.~Hatami, J.~Hirst and S.~Norine.
\newblock The inducibility of blow-up graphs.
\newblock {\em Journal of Combinatorial Theory, Series B}, {\bfseries
  109}:196--212, 2014.

\bibitem{huynh2021subgraph}
T.~Huynh, G.~Joret and D.~R. Wood.
\newblock Subgraph densities in a surface. \textit{arXiv:2003.13777} preprint,
  2021.

\bibitem{kral}
D.~Král', S.~Norin and J.~Volec.
\newblock A bound on the inducibility of cycles.
\newblock {\em Journal of Combinatorial Theory, Series A}, {\bfseries
  161}:359--363, 2019.

\bibitem{liu2021homomorphism}
C.-H. Liu.
\newblock Homomorphism counts in robustly sparse graphs.
  \textit{arXiv:2107.00874} preprint, 2021.

\bibitem{pippenger}
N.~Pippenger and M.~C. Golumbic.
\newblock The inducibility of graphs.
\newblock {\em Journal of Combinatorial Theory, Series B}, {\bfseries
  19}(3):189--203, 1975.

\bibitem{wood}
D.~R. Wood.
\newblock On the maximum number of cliques in a graph.
\newblock {\em Graphs and Combinatorics}, {\bfseries 23}:337--352, 2007.

\bibitem{wormald}
N.~C. Wormald.
\newblock On the frequency of 3-connected subgraphs of planar graphs.
\newblock {\em Bulletin of the Australian Mathematical Society}, {\bfseries
  34}(2):309–317, 1986.

\bibitem{yuster}
R.~Yuster.
\newblock On the exact maximum induced density of almost all graphs and their
  inducibility.
\newblock {\em Journal of Combinatorial Theory, Series B}, {\bfseries
  136}:81--109, 2019.

\end{thebibliography}
\end{document}